\documentclass[reqno, 12pt]{amsart}
\usepackage{pifont}
\usepackage{amsfonts}
\usepackage{}
\usepackage{mathrsfs}
\usepackage{amssymb,url}
\usepackage{cite}
\usepackage{geometry}
\date{}
\allowdisplaybreaks[4]
\theoremstyle{plain}
\newtheorem{thm}{Theorem}[section]

\newtheorem{lem}[thm]{Lemma}
\newtheorem{pro}[thm]{Proposition}

\theoremstyle{remark}

\theoremstyle{definition}

\newcommand{\tr}{\mathbb{R}}
\def\sln{\operatorname{SL}(n)}
\def\gln{\operatorname{GL}(n)}
\def\lin{\operatorname{lin}}

\usepackage[colorlinks,linkcolor=blue,urlcolor=red,linktocpage]{hyperref}

\makeatletter
    \newcommand{\rmnum}[1]{\romannumeral #1}
    \newcommand{\Rmnum}[1]{\expandafter\@slowromancap\romannumeral #1@}
  \makeatother

\numberwithin{equation}{section}

\begin{document}
\title
{$L_p$ Minkowski Valuations on polytopes}

\author[]{Jin Li$^{1,2}$}
\author[]{Gangsong Leng$^1$}
\address[address1]{Department of Mathematics, Shanghai University, Shanghai 200444, China}
\address[address2]{Institut f\"{u}r Diskrete Mathematik und Geometrie, TU Wien, Wien 1040 Austria}

\email[Jin Li]{\href{mailto: Jin Li
<lijin2955@gmail.com>}{lijin2955@gmail.com}}
\email[Gangsong Leng]{\href{mailto: Gangsong Leng
<lenggangsong@163.com>}{lenggangsong@163.com}}

\begin{abstract}
For $1 \leq p < \infty$, Ludwig, Haberl and Parapatits classified $L_p$ Minkowski valuations intertwining the special linear group with additional conditions such as homogeneity and continuity. In this paper,a complete classification of $L_p$ Minkowski valuations intertwining the special linear group on polytopes without any additional conditions is established for $p \geq 1$ including $p = \infty$. For $n=3$ and $p=1$, there exist valuations not mentioned before.
\end{abstract}

\subjclass[2010]{52A20, 52B45}

\keywords{$L_\infty$ Minkowski valuation, $L_\infty$ projection body, $L_p$ Minkowski valuation, function-valued valuation, $\sln$ contravariant, $\sln$ covariant}

\thanks{}

\maketitle
\section{Introductions}
Let $\mathcal {K}_o ^n$ be the set of convex bodies (i.e., compact convex sets) in $\mathbb{R}^n$ containing the \text{origin}, $\mathcal {P}_o ^n$ the set of polytopes in $\mathbb{R}^n$ containing the origin and $\mathcal {T}_o^n$ the set of simplices in $\mathbb{R}^n$ \text{containing} the origin as one of their vertices.

For $1 \leq p \leq \infty$ and $K,L \in \mathcal {K}_o ^n$, the \emph{$L_p$ Minkowski sum} of $K$ and $L$ is defined by its support function as
\begin{align}\label{205}
h_{K +_p L}(x)= (h_K (x) ^p + h_L (x)^p)^{1/p}, ~~x \in \mathbb{R}^n.
\end{align}
Here $h_K$ is the support function of $K$; see Section \ref{s2}. When $p = \infty$, the definition (\ref{205}) should be interpreted as $h_{K +_\infty L}(x) = h_K (x) \vee h_L (x)$, the maximum of $h_K (x)$ and $h_L (x)$. When $p=1$, the definition (\ref{205}) gives the ordinary Minkowski addition.

An \emph{$L_p$ Minkowski valuation} is a function $Z: \mathcal {P}_o ^n \rightarrow \mathcal {K}_o ^n$ such that
\begin{align}\label{107}
    Z(K \cup L) +_p Z(K \cap L) = ZK +_p ZL,
\end{align}
whenever $K,L,K \cup L,K \cap L \in \mathcal {P}_o ^n$. In some cases, we will just consider valuations defined on $\mathcal {T}_o ^n$ that means (\ref{107}) holds whenever $K,L,K \cup L,K \cap L \in \mathcal {T}_o ^n$.

For $1 \leq p < \infty$, Ludwig \cite{Lud05}, Haberl \cite{Hab12b} and Parapatits \cite{Par14a}, \cite{Par14b} classified $L_p$ Minkowski valuations intertwining the special linear group, $\sln$, with some additional conditions such as homogeneity and continuity.

A map $Z$ from $\mathcal{K}_o ^n$ to the power set of $\mathbb{R}^n$ is called $\sln$ contravariant if
$$Z (\phi K) = \phi ^{-t} ZK$$
for any $K \in \mathcal{K}_o ^n$ and any $\phi \in \sln$. The map $Z$ is called $\sln$ covariant if
$$Z (\phi K) = \phi ZK$$
for any $K \in \mathcal{K}_o ^n$ and any $\phi \in \sln$.
Notice that $\{ o \}$ is the only subset of $\mathbb{R}^n$ invariant under all $\sln$ transforms. Thus if $Z$ is $\sln$ contravariant (or covariant), then
\begin{align}\label{37}
Z \{ o \} = \{ o \}.
\end{align}

Generalizing results for homogeneous or translation invariant valuations by Ludwig \cite{Lud02b,Lud05}, Haberl \cite{Hab12b} and Parapatits \cite{Par14a}, \cite{Par14b} established the following classification theorem.
\begin{thm}[Haberl \cite{Hab12b} and Parapatits \cite{Par14a}]\label{Lpcontra}
Let $n \geq 3$. A map $Z : \mathcal{P}_o ^n \to \mathcal {K}_o ^n$ is an $\sln$ contravariant Minkowski valuation if and only if there exist constants
$c_1,c_2,c_3 \in \mathbb{R}$ with $c_1 \geq 0$ and $c_1 +c_2 +c_3 \geq 0$ such that
$$ZP = c_1 \Pi P + c_2 \Pi_o P + c_3 \Pi_o (-P)$$
for every $P \in \mathcal{P}_o ^n$.

For $1 < p < \infty$, a map $Z : \mathcal{P}_o ^n \to \mathcal {K}_o ^n$ is an $\sln$ contravariant $L_p$ Minkowski valuation if and only if there exist constants
$c_1,c_2 \geq 0$ such that
$$ZP = c_1 \hat{\Pi}_p^+ P +_p c_2 \hat{\Pi}_p ^- P$$
for every $P \in \mathcal{P}_o ^n$.
\end{thm}
Here $\Pi$ is the classical \emph{projection body}, while $\hat{\Pi}_p^+$ and $\hat{\Pi}_p^-$ are the \emph{asymmetric $L_p$ projection bodies} first defined in \cite{Lud05}; see Section \ref{s2}. $\Pi_o$ is a valuation defined by $h_{\Pi_o P} = h_{\Pi P} - h_{\hat{\Pi}^+ P}$.

\begin{thm}[Haberl \cite{Hab12b} and Parapatits \cite{Par14b}]\label{lpco}
Let $n \geq 3$, $1 \leq p < \infty$ and $\{e_i\}_{i=1}^n$ be the standard basis of $\tr^n$. A map $Z : \mathcal{P}_o ^n \to \mathcal {K}_o ^n$ is an $\sln$ covariant $L_p$ Minkowski valuation which is continuous at the line segment $[o,e_1]$ if and only if there exist constants
$c_1,\dots,c_4 \geq 0$ such that
$$ZP = c_1 M_p^+ P +_p c_2 M_p ^- P +_p c_3 P +_p c_4 (-P)$$
for every $P \in \mathcal{P}_o ^n$.
\end{thm}
Here $M_p^+$, $M_p^-$ are the \emph{asymmetric $L_p$ moment bodies} first defined in \cite{Lud05}; see Section \ref{s2}.

Haberl and Schuster \cite{HS09a} established affine isoperimetric inequalities for asymmetric $L_p$ projection bodies and asymmetric $L_p$ moment bodies. For other results on $L_p$ Minkowski valuations, see \cite{abardia2011p,abardia2013m,Lud03,Lud10b,PS12,SS06,Sch08,Sch2010,SW12,Lud12,Ober2014,Tsa12,Wan11}.
$L_p$ projection bodies and $L_p$ moment bodies ($1<p<\infty$) were first studied in \cite{LYZ00a} as part of $L_p$ Brunn-Minkowski theory developed by Lutwak, Yang, and Zhang, and many others; see \cite{HS09b,Lut93,Lut96,LYZ00b,LYZ02a,LYZ04a,LYZ05}.

As first result of this paper, we establish a classification of $L_\infty$ Minkowski valuations. We remark that the $L_\infty$ sum of $K,L \in \mathcal {K}^n$ is equal to its convex hull, $[K,L]$.

\begin{thm}\label{thm1.1}
Let $n \geq 3$. A map $Z : \mathcal{P}_o ^n \to \mathcal {K}_o ^n$ is an $\sln$ contravariant $L_\infty$ Minkowski valuation if and only if there exist constants $c_1, c_2 \geq 0$ such that
$$ZP = c_1 \hat{\Pi}_\infty ^+ P +_\infty c_2 \hat{\Pi}_\infty ^- P$$
for every $P \in \mathcal{P}_o ^n$.
\end{thm}

The \emph{asymmetric $L_\infty$ projection body} $\hat{\Pi}_\infty ^+ :\mathcal {P}_o ^n \to \mathcal {K}_o ^n$ is defined by
\begin{align*}
\hat{\Pi} _\infty ^+ P = \left[ o,\frac{u_i}{h_{P}(u_i)}: u_i \in \mathcal {N}(P) \setminus \mathcal {N}_o(P) \right],
\end{align*}
and
\begin{align*}
\hat{\Pi}_\infty ^- P = - \hat{\Pi}_\infty ^+ P.
\end{align*}
Here $\mathcal {N} (P)$ is the set of outer unit normals to facets (that is $n-1$ dimensional faces) of $P$ and $\mathcal {N}_o (P)$ is the set of outer unit normals to facets of $P$ which contain the origin.
Both $\hat{\Pi} _\infty ^+$ and $\hat{\Pi} _\infty ^-$ are the limits of $\hat{\Pi} _p ^+$ and $\hat{\Pi} _p ^-$ as $p \to \infty$. So they are clearly $L_\infty$ Minkowski valuations. Also, $\hat{\Pi}_\infty^+ $ is an extension of the polarity. Indeed,
if a convex body $K$ contains the origin in its interior, then
$\hat{\Pi} _\infty ^+ K = K^*$, the polar body of $K$. All the details can be found in Section \ref{s2}.

If a valuation $Z_p$ is an $L_p$ Minkowski valuation, then the limit $\lim\limits_{p \to \infty} Z_p$ is an $L_\infty$ Minkowski valuation. But there could be more $L_\infty$ Minkowski valuations than the limits of $L_p$ cases. Indeed, Theorem \ref{thm1.3} shows that there are additional examples.

\begin{thm}\label{thm1.3}
Let $n \geq 3$. A map $Z : \mathcal{P}_o ^n \to \mathcal {K}_o ^n$ is an $\sln$ covariant $L_\infty$ Minkowski valuation if and only if there exist constants $0 \leq a_1 \leq \dots \leq a_n$, $0 \leq b_1 \leq \dots \leq b_n$ such that
\begin{align*}
ZP = a_d P +_\infty (- b_d P)
\end{align*}
for every $d$-dimensional convex polytope $P \in \mathcal{P}_o ^n$, $1 \leq d \leq n$, while $Z \{ o \} = \{ o \}$.
\end{thm}

If $\dim P =n$, then $\lim\limits_{p \to \infty} M_p ^+ P = P$ and $\lim\limits_{p \to \infty} M_p ^- P = -P$. If $\dim P < n$, $\lim\limits_{p \to \infty} M_p ^+ P = \{o \}$ and $\lim\limits_{p \to \infty} M_p ^- P = \{o \}$. This is the reason that $\lim\limits_{p \to \infty} M_p ^+$ and $\lim\limits_{p \to \infty} M_p ^-$ do not show up in Theorem \ref{thm1.3}; see Section \ref{s2} for details. In Theorem \ref{thm1.3}, we do not have any continuity assumptions. It inspires us to also find a classification result for $\sln$ covariant $L_p$ Minkowski valuations without any continuity assumptions for finite $p$.
\begin{thm}\label{thm1.4}
Let $n \geq 3$ and $1 < p < \infty$. A map $Z : \mathcal{P}_o ^n \to \mathcal {K}_o ^n$ is an $\sln$ covariant $L_p$ Minkowski valuation if and only if there exist constants
$c_1,\dots,c_4 \geq 0$ such that
$$ZP = c_1 M_p^+ P +_p c_2 M_p ^- P +_p c_3 P +_p c_4 (-P) $$
for every $P \in \mathcal{P}_o ^n$.
\end{thm}

\begin{thm}\label{thm1.6}
Let $n \geq 4$. A map $Z : \mathcal{P}_o ^n \to \mathcal {K}_o ^n$ is an $\sln$ covariant Minkowski valuation if and only if there exist constants $c_1,\dots,c_4 \geq 0$ such that
$$ZP = c_1 M^+ P + c_2 M ^- P +  c_3 P + c_4 (-P)$$
for every $P \in \mathcal{P}_o ^n$.
\end{thm}

\begin{thm}\label{thm1.7}
A map $Z : \mathcal{P}_o ^3 \to \mathcal {K}_o ^3$ is an $\sln$ covariant Minkowski valuation if and only if there exist constants
$a_1,a_2,b_1,b_2,c_1,c_2 \geq 0$ satisfying $a_1 \leq a_2$, $b_1 \leq b_2$, $a_2-a_1 \leq b_2$ and $b_2-b_1 \leq a_2$ such that
$$ZP = c_1 M^+ P + c_2 M ^- P + D_{a_1,a_2,b_1,b_2} P$$
for every $P \in \mathcal{P}_o ^3$.
\end{thm}
The convex body $D_{a_1,a_2,b_1,b_2} P$ is a generalization of the difference body. We remark that it was omitted in the classifcation by Ludwig \cite[Theorem 1]{Lud05}.
Denote by $\mathcal {E}_o (P)$ the set of edges of $P$ that contain the origin and by $\mathcal {F}_o (P)$ the set of $2$-dimensional faces of $P$ that contain the origin.
For $P \in \mathcal {P}_o ^3$,
\begin{align*}
h_{D_{a_1,a_2,b_1,b_2} P}
&= a_1 h_P + (a_2-a_1) \sum_{F \in \mathcal {F}_o (P)}h_{F} - (a_2-a_1) \sum_{E \in \mathcal {E}_o (P)}h_{E} \\
&\qquad + b_1 h_{-P} + (b_2-b_1) \sum_{F \in \mathcal {F}_o (P)}h_{-F} - (b_2-b_1) \sum_{E \in \mathcal {E}_o (P)}h_{-E}
\end{align*}
if $\dim P =3$;
\begin{align*}
h_{D_{a_1,a_2,b_1,b_2} P}
&= (2a_2-a_1) h_P - (a_2-a_1) \sum_{E \in \mathcal {E}_o (P)}h_{E} \\
&\qquad +(2b_2-b_1) h_{-P} - (b_2-b_1) \sum_{E \in \mathcal {E}_o (P)}h_{-E}
\end{align*}
if $\dim P =2$; and
\begin{align*}
h_{D_{a_1,a_2,b_1,b_2} P}
&= a_1 h_P + b_1 h_{-P}
\end{align*}
if $\dim P =1$.
That $h_{D_{a_1,a_2,b_1,b_2} P}$ is a support function is guaranteed by the conditions on $a_1,a_2,b_1,b_2$.

Theorem \ref{thm1.4}, \ref{thm1.6} and \ref{thm1.7} are based on the classification of function-valued valuations (Lemma \ref{lem5.4}).
The map $Z: \mathcal{P}_o ^n \to \mathcal{K}_o ^n$ is an $L_p$ Minkowski valuation if and only if $\Phi : P \mapsto h_{ZP}^p$ is a function-valued valuation; see Section \ref{Slp} for more details.
There exist additional \emph{complicated} function-valued valuations ($P \mapsto \Phi_{p;a_1,a_2}P + \Phi_{p;b_1,b_2}(-P)$; see the definition in Section \ref{Slp}) if we do not assume continuity like Haberl \cite{Hab12b} and Parapatits \cite{Par14b} did. However, in generally, they are not $L_p$ Minkowski valuations for $p>1$.
For $p=1$, $h_{D_{a_1,a_2,b_1,b_2} P} = \Phi_{1;a_1,a_2}P + \Phi_{1;b_1,b_2}(-P)$ for $\dim P \leq 3$. For $n \geq 4$, $P \mapsto \Phi_{1;a_1,a_2}P + \Phi_{1;b_1,b_2}(-P)$ is also a function-valued valuation on $\mathcal {P}_o^n$.  But the example used for $n \geq 4$ and $p = 1$ in Lemma \ref{lem5.1} shows that $\Phi_{1;a_1,a_2}[-e_1,e_1,e_2,e_3,e_4] + \Phi_{1;b_1,b_2}(-[-e_1,e_1,e_2,e_3,e_4])$ is not a support function.
That means $D_{a_1,a_2,b_1,b_2}$ even cannot be extended to simplices that contain the origin in one of their edges for dimension greater than or equal to $4$.
However, Theorem \ref{thm1.5} shows that $D_{a_1,a_2,b_1,b_2}$ can be extended to a valuation on $\mathcal {T}_o^n$ also for $n \geq 4$.


\section{Preliminaries and Notation}\label{s2}
Let $\mathbb{R}^n$ be the $n$-dimensional Euclidean space and $\{e_i\}_{i=1}^n$ its standard basis. For $1 \leq d \leq n-1$, we will also use $\mathbb{R}^d$ to denote the linear space spanned by $\{ e_1, \dots, e_d \}$. The usual scalar product of two vectors $x,y \in \mathbb{R}^n$ shall be denoted by $x \cdot y$. The convex hull of a set $A \subset \mathbb{R}^n$ is denoted by $[A]$.

Let $a,b \in \mathbb{R}$. We write $a \vee b := \max \{a,b \}$.

Let $\mathcal {K}^n$ be the set of convex bodies in $\mathbb{R}^n$. For $K \in \mathcal {K}^n$, $\text{relint}\,K$, $\text{relbd}\, K$, $K^c$ and $\lin K$ denote the relative interior, the relative boundary, the relative complement with respect to the affine hull of $K$, and the linear hull of $K$, respectively. We mention that $\text{relint}\,K \neq \emptyset$ if $K \neq \emptyset$.

Let $Gr(n,j)$ be the set of $j$-dimensional linear subspaces in $\mathbb{R}^n$.
For $x \in \mathbb{R}^n$, $A \subset \mathbb{R}^n$, $V \in Gr(n,j)$, let $x|V$ be the orthogonal projection of $x$ onto $V$ and $A|V = \{x|V: x \in A \}$. We also write $x|K$ for the orthogonal projection of $x$ onto the linear hull of $K \in \mathcal {K}_o^n$.

The \emph{support function} of a convex body $K$ is defined by
\begin{align*}
h_K(x) = \max \{ x \cdot y : y \in K\}
\end{align*}
for any $x \in \mathbb{R}^n$. The support function is sublinear, i.e., it is homogeneous,
\begin{align*}
h_K(\lambda x) = \lambda h_K(x)
\end{align*}
for any $x \in \mathbb{R}^n$, $\lambda \geq 0$, and subadditive,
\begin{align*}
h_K(x+y) \leq h_K(x) + h_K(y)
\end{align*}
for any $x,y \in \mathbb{R}^n$. The support function is also continuous on $\mathbb{R}^n$ by its convexity. A convex body is uniquely determined by its support function, and for any sublinear function $h$, there exists a convex body $K$ such that $h_K = h$. It is easy to see that
\begin{align}\label{210}
h_{\lambda K} = \lambda h_K
\end{align}
for any $\lambda \geq 0$ and $K \in \mathcal {K}^n$. Also,
\begin{align*}
h_{\phi K} (x) = h_K (\phi^t x)
\end{align*}
for $x \in \tr^n$, $\phi \in \gln$ and $K \in \mathcal {K}^n$.

For $K, L \in \mathcal {K}^n$, if $K \cup L$ is convex, then
$$h_{K \cup L} = \max \{h_K, h_L \}, ~~h_{K \cap L} = \min \{h_K, h_L \}.$$
Hence the identity map is an $L_p$ Minkowski valuation on $\mathcal {K}^n$ (or on $\mathcal {P}_o^n$).

The \emph{face} of $K \in \mathcal {K}^n$ with normal vector $u \in S^{n-1}$ is $F(K,u) = \{y \in K: y \cdot u = h_K(u) \}$.

A \emph{hyperplane} $H$ through the origin with a normal vector $u$ is defined by $\{ x \in\mathbb{R}^n : x \cdot u = 0 \}$. Furthermore define $H^- := \{ x \in\mathbb{R}^n : x \cdot u \leq 0 \}$ and $H^+ := \{ x \in\mathbb{R}^n : x \cdot u \geq 0 \}$. For $0 < \lambda < 1$, let $H_\lambda$ be the hyperplane through the origin with normal vector $(1-\lambda) e_1- \lambda e_2$.

The following $\sln$ transforms $\phi_1,\phi_2,\phi_3,\phi_4$ depending on $\lambda$, $0 < \lambda < 1$, will be useful.
$$\phi _1 e_1 = \lambda e_1 + (1-\lambda) e_2,~\phi _1 e_2 = e_2,~\phi _1 e_n = \frac{1}{\lambda} e_n,~\phi _1 e_i = e_i,~\text{for}~3 \leq i \leq n-1,$$
$$\phi _2 e_1 = e_1,~\phi _2 e_2 = \lambda e_1 + (1-\lambda) e_2,~\phi _2 e_n = \frac{1}{1-\lambda} e_n,~\phi _2 e_i = e_i,~\text{for}~3 \leq i \leq n-1,$$
$$\phi_3 e_1 = (\frac{1}{\lambda})^{1/n} (\lambda e_1 + (1-\lambda) e_2),~\phi_3 e_2 = (\frac{1}{\lambda})^{1/n} e_2,~\phi_3 e_i = (\frac{1}{\lambda})^{1/n} e_i,~\text{for}~3 \leq i \leq n,$$
and
$$\phi_4 e_1 = (\frac{1}{1-\lambda})^{1/n} e_1,~\phi_4 e_2 = (\frac{1}{1-\lambda})^{1/n} (\lambda e_1 + (1-\lambda) e_2),
\phi_4 e_i = (\frac{1}{1-\lambda})^{1/n} e_i,~\text{for}~3 \leq i \leq n.$$

For $1 \leq d \leq n$, let $T^d = [o,e_1,e_2,e_3, \dots,e_d]$ and $\hat{T}^{d-1} = [o,e_1,e_3 \dots,e_d]$. Hence, for $s > 0$,
$sT^{d}\cap H_\lambda ^- = \phi _1 sT^{d}$, $sT^{d}\cap H_\lambda ^+ = \phi _2 sT^{d}$ and $sT^d \cap H_\lambda =\phi _1 s\hat{T}^{d-1}$ for $2 \leq d \leq n-1$. Also,
$sT^{n}\cap H_\lambda ^- = \phi_3 \lambda^{1/n} sT^{n}$, $sT^{n}\cap H_\lambda ^+ = \phi_4 (1-\lambda)^{1/n}sT^{n}$ and $sT^n \cap H_\lambda =\phi _1 \lambda^{1/n}s\hat{T}^{n-1}$.

The \emph{asymmetric $L_p$ moment body} of a star body $K$ is defined by
\begin{align*}
h_{M_p^+ K} (x) = \left(\int_K (\max \{x \cdot y,0\})^p dy \right)^{1/p},~~x \in \mathbb{R}^n
\end{align*}
and
\begin{align*}
h_{M_p^- K} (x) = \left(\int_K (\max \{-x \cdot y,0\})^p dy \right)^{1/p},~~x \in \mathbb{R}^n.
\end{align*}
Both $M_p^+,M_p^-$ are $\sln$ covariant $L_p$ Minkowski valuations. Positive combinations of $M_p^+$ and $M_p^-$ were first characterized as $(\frac{n}{p}+1)$-homogeneous and $\sln$ covariant $L_p$ Minkowski valuations by Ludwig \cite{Lud05}. Also see Theorem \ref{lpco}.
For $\dim K = n$,
\begin{align*}
h_{M_\infty^+ K} (x) = \lim\limits_{p \to \infty} h_{M_p^+ K} (x) = \max_{y \in K} \{x \cdot y \} = h_K (x),~~x \in \mathbb{R}^n
\end{align*}
and for $\dim K <n$, $M_\infty^+ K = \{o\}$.


The \emph{projection body} of $K \in \mathcal {K}^n$ is defined by
\begin{align*}
h_{\Pi K} (x) = \frac{1}{2} \int_{S^{n-1}} |x \cdot u| dS_K(u), ~~x \in \mathbb{R}^n,
\end{align*}
where $S_K$ is the \emph{surface area measure} of $K$. For a Borel set $\omega \subset S^{n-1}$, $S_K(\omega)$ is the $(n-1)$-Hausdorff measure of $\{x \in \operatorname*{bd} K: \nu_K(x) \in \omega \}$, where $\nu_K(x)$ are outer normal vectors to $K$ at $x$.

The \emph{cone-volume measure} of $K \in \mathcal{K}_o^n$ is defined by $dv_K (u) = h_K (u) d S_K(u)$. The \emph{asymmetric $L_p$ projection body} of $P \in \mathcal {P}_o ^n$ is defined by
\begin{align*}
h_{\hat{\Pi}_p^+ P} (x) = \left(\int_{S^{n-1} \setminus \mathcal {N}_o (P)} (\frac{\max \{x \cdot u,0\}}{h_P(u)})^p dv_P(u)\right)^{1/p}
\end{align*}
for any $x \in \mathbb{R}^n$ and
\begin{align*}
h_{\hat{\Pi}_p^- P} (x) = \left(\int_{S^{n-1} \setminus \mathcal {N}_o (P)} (\frac{\max \{-x \cdot u,0\}}{h_P(u)})^p dv_P(u)\right)^{1/p} = h_{\hat{\Pi}_p^+ P} (-x)
\end{align*}
for any $x \in \mathbb{R}^n$. Positive combinations of $\hat{\Pi}_p^+$ and $\hat{\Pi}_p^-$ were first characterized as $(\frac{n}{p}-1)$-homogeneous, $\sln$ contravariant $L_p$ Minkowski valuations by Ludwig \cite{Lud05}. Also see Theorem \ref{Lpcontra}. For $p=1$, $\Pi_o$ defined by $h_{\Pi_o P} = h_{\Pi P} - h_{\hat{\Pi}^+ P}$ is an additional valuation.


When $p \to \infty$, we have
$$\lim\limits_{p \to \infty} h_{\hat{\Pi}_p^+ P} (x) = \max_{u_i \in \mathcal{N} (P) \setminus \mathcal{N}_o (P) } \{\frac{x \cdot u_i}{h_P(u_i)},0\} = h_{\hat{\Pi}_\infty^+ P} (x).$$
Hence $\hat{\Pi}_\infty^+$ is a $(-1)$-homogeneous, $\sln$ contravariant $L_\infty$ Minkowski valuation. For $K \in \mathcal {K}^n$ containing the origin in its interior,
\begin{align*}
\lim\limits_{p \to \infty} h_{\hat{\Pi}_p^+ K} (x) = \lim\limits_{p \to \infty} \left(\int_{S^{n-1}} (\frac{\max \{x \cdot u,0\}}{h_K(u)})^p dv_K(u)\right)^{1/p} = \operatorname*{ess~sup}_{u \in S^{n-1}} \frac{x \cdot u}{h_K(u)}.
\end{align*}
Here the essential supremum is with respect to the cone-volume measure. We have
\begin{align*}
\frac{x \cdot u}{h_K(u)} = \frac{1}{\rho_K(x)} \frac{\rho_K(x) x \cdot u}{h_K(u)} \leq  \frac{1}{\rho_K(x)} \frac{\rho_K(x) x \cdot u}{\rho_K(x) x \cdot u} = \frac{1}{\rho_K(x)},
\end{align*}
where equality holds when $h_K(u) = \rho_K(x) x \cdot u$. Here $\rho_K(x):= \max \{\lambda>0: \lambda x \in K \}$ is the radial function of $K$. Also since there exists a normal vector $u$ at $\rho_K(x) x$ such that $u \in \text{supp}~v_K$, the support set of $v_K$, and $u \mapsto \frac{x \cdot u}{h_K(u)}$ is continuous,
$$h_{\hat{\Pi}_\infty^+ K} (x) = \operatorname*{ess~sup}_{u \in S^{n-1}} \frac{x \cdot u}{h_K(u)} = \frac{1}{\rho_K(x)} = h_{K^*}(x).$$

%
%

The following lemma will be used to classify $L_\infty$ Minkowski valuations. It is an $L_\infty$ version of the Cauchy functional equation.
\begin{lem}\label{cauchy}
If a function $f:(0,\infty) \to [0,\infty)$ satisfies
\begin{align}\label{52b}
f(x+y) \vee a = f(x) \vee f(y),
\end{align}
for any $x,y > 0$, where $a \geq 0$ is a constant, then
\begin{align*}
  f(z) = f(1) \geq a
\end{align*}
for any $z>0$.
\end{lem}
\begin{proof}
For $x=y=1$ in (\ref{52b}), we directly get $f(1) \geq a$.
We will prove $f(z) = f(1)$ in two steps.

Step \ding {172}: Let $k$ be an integer. We will show, by induction, that
\begin{align}\label{b1}
f(2^k) = f(1).
\end{align}

The case $k=0$ is trivial.
Taking $x=y=2^k$ in (\ref{52b}), we get
\begin{align}\label{b3}
f(2^{k+1}) \vee a = f(2^k) \vee f(2^k).
\end{align}
for any integer $k$.
Hence $$a \leq f(2^k)$$
for any $k$.
For $k \geq 1$, assume that (\ref{b1}) holds for $k-1$.
By (\ref{b3}),
if $a<f(1)$, we have $$f(2^k) = f(2^{k-1}) = f(1);$$
if $a=f(1)$, we have $$f(2^k) \leq f(2^{k-1}) = f(1) = a \leq f(2^k).$$
Thus, (\ref{b1}) holds for $k \geq 1$.

For $k \leq -1$, assume that (\ref{b1}) holds for $k+1$.
Since (\ref{b3}) and $a \leq f(1)$, we have $$f(2^k) = f(2^{k+1}) = f(1).$$
Thus we obtain that (\ref{b1}) holds for any integer $k$.

Step \ding {173}: Let $z >0$. There exists an integer $k$ such that $2^k \leq z < 2^{k+1}$.
Taking $x+y=2^{k+1}$, $x = z$ in (\ref{52b}), we obtain that
\begin{align*}
f(2^{k+1}) \vee a = f(z) \vee f(2^{k+1} - z).
\end{align*}
Since $a \leq f(1)$ and
$f(2^{k+1})= f(1)$ (step \ding {172}), we have
\begin{align}\label{33a}
f(z) \leq f(1).
\end{align}
for any $z >0$.

We assume $z \neq 2^k$.
If $a < f(1)$, taking $x+y=z$, $x=2^k$ in (\ref{52b}),
we obtain that
\begin{align*}
f(z) \vee a = f(2^k) \vee f(z-2^k).
\end{align*}
By (\ref{33a}), $f(z-2^k) \leq f(1)$.
Also since $f(2^k) = f(1)$ from step \ding {172},
we have $$f(z) = f(1).$$
If $a = f(1)$, taking $x=y = z$ in (\ref{52b}), we get
\begin{align*}
f(2z) \vee a = f(z) \vee f(z).
\end{align*}
Then, we have
$$f(1) = a \leq f(z) \leq f(1).$$
The proof is complete.
\end{proof}

The following statements will be used to determine $L_\infty$ Minkowski valuations by their values on $\mathcal {T}_o^n$.

Define $\mathcal{P}_1 := \mathcal {T}_o^n$ and $\mathcal{P}_i := \mathcal{P}_{i-1} \cup \{ P_1 \cup P_2 \in \mathcal {P}_o^n: P_1,P_2 \in \mathcal{P}_{i-1} ~\text{with~disjoint~relative~interiors}\}$ recursively. Note that for any $P \in \mathcal {P}_o^n$, there exists an $i$ such that $P \in \mathcal{P}_{i}$.

Let $H \subset \mathbb{R}^n$ be a hyperplane through the origin. For any $P \in \mathcal{P}_{i}$, $i \geq 1$, we also have
\begin{align}\label{303}
P \cap H \in \mathcal{P}_{i}.
\end{align}
Indeed, for any $T \in \mathcal {T}_o^n$, we have $T \cap H \in \mathcal {T}_o^n$. Assume that for any $P \in \mathcal{P}_{i-1}$, $i \geq 2$, we have $P \cap H \in \mathcal{P}_{i-1}$. Then for any $P = P_1 \cup P_2$, where $P_1,P_2 \in \mathcal{P}_{i-1}$ have disjoint relative interiors, we have
$$P \cap H = (P_1 \cap H) \cup (P_2 \cap H).$$
If $P_1 \cap H$ and $P_2 \cap H$ have disjoint relative interiors, then $P \cap H \in \mathcal{P}_{i}$. Otherwise, only two possibilities could happen: $(P_1 \cap H) \subset (P_2 \cap H)$ and $(P_2 \cap H) \subset (P_1 \cap H)$. For both possibilities, we have $P \cap H \in \mathcal{P}_{i-1} \subset \mathcal{P}_i$.

\section{$\sln$ contravariant $L_\infty$ Minkowski valuations}
In this section, we first show that any $\sln$ contravariant $L_\infty$ \text{Minkowski} valuation on $\mathcal{T}_o ^n$ vanishes on lower dimensional simplices in $\mathcal{T}_o ^n$.

\begin{lem}\label{lem3.1}
Let $n \geq 3$. If $Z : \mathcal{T}_o ^n \to \mathcal{K}_o ^n$ be an $\sln$ contravariant $L_\infty$ Minkowski valuation, then $ZT= \{o \}$ for any $T \in \mathcal{T}_o ^n$ satisfying $\dim T < n$.
\end{lem}
\begin{proof}
Let $T \in \mathcal{T}_o ^n$ and $\dim T =d <n$. We can assume (w.l.o.g.) that the linear hull of $T$ is $ \lin \{ e_1, \dots, e_d \}$, the linear space spanned by $ \{ e_1, \dots, e_d \}$. Let $\phi := \left[ {\begin{array}{*{20}{c}}
I&A \\
0&B
\end{array}} \right] \in \sln$, where $I \in \mathbb{R}^{d \times d}$ is the identity matrix, $A \in \mathbb{R}^{d \times (n-d)}$ is an arbitrary matrix, $B \in \mathbb{R}^{(n-d) \times (n-d)}$ is a matrix with det $B = 1$, $0 \in \mathbb{R}^{(n-d)\times d}$ is the zero matrix. Also let $x = \left( {\begin{array}{*{20}{c}}
{x'}\\
{x''}
\end{array}} \right) \in \mathbb{R}^{d \times (n-d)}$ and $x'' \neq 0$. Then $\phi T = T$, and with the $\sln$ contravariance of $Z$, we have
\begin{align*}
h_{ZT} (x) = h_{Z \phi T} (x) = h_{ZT}(\phi ^{-1} x) = h_{ZT} \left( {\begin{array}{*{20}{c}}
{x' - AB^{-1}x''}\\
{B^{-1}x''}
\end{array}} \right).
\end{align*}

For $d \leq n-2$, we can choose an suitable matrix $B$ such that $B^{-1}x''$ be any nonzero vector on $\lin \{ e_{d+1},\dots,e_n \}$. After fixing $B$ we can also choose an suitable matrix $A$ such that $x' - AB^{-1}x''$ is any vector in $\lin \{ e_1, \dots, e_d \}$. So $h_{ZT} (x)$ is constant on a dense set of $\mathbb{R}^n$. By the continuity of the support function, we get $h_{ZT} (x) = 0$.

For $d = n-1$, $B=1$. We can choose an suitable $A$ such that $x' - AB^{-1}x''=0$, and then $h_{ZT} (x) = h_{ZT} (x_n e_n)$, where $x_n$ is the $n$-th coordinate of $x$. By the $\sln$ contravariance of $Z$, we only need to show that $h_{Z(sT^{n-1})} (x) = h_{Z(sT^{n-1})} (x_ne_n)= 0$ for any $s > 0$.

For $0 < \lambda < 1$, define $H_\lambda$ and $\phi _1,\phi_2 \in \sln$ as in Section \ref{s2}. Since $Z$ is a valuation,
\begin{align*}
h_{Z(sT^{n-1})} (e_n) \vee h_{Z(sT^{n-1} \cap H_\lambda)} (e_n) = h_{Z(sT^{n-1}\cap H_\lambda ^-)} (e_n) \vee h_{Z(sT^{n-1} \cap H_\lambda ^+)} (e_n).
\end{align*}
From the conclusion above for $d=n-2$, we get
\begin{align*}
h_{Z(sT^{n-1})} (e_n) = h_{Z(sT^{n-1}\cap H_\lambda ^-)} (e_n) \vee h_{Z(sT^{n-1} \cap H_\lambda ^+)} (e_n).
\end{align*}
Also by the $\sln$ contravariance of $Z$, we obtain
\begin{align*}
h_{Z(sT^{n-1})} (e_n) &= h_{Z(\phi _1 sT^{n-1})} (e_n) \vee h_{Z(\phi _2 sT^{n-1})} (e_n) \\
&= h_{Z( sT^{n-1})} (\phi _1 ^{-1} e_n) \vee h_{Z (sT^{n-1})} (\phi _2 ^{-1} e_n) \\
&= h_{Z( sT^{n-1})} (\lambda e_n) \vee h_{Z (sT^{n-1})} ((1-\lambda) e_n).
\end{align*}

If $h_{Z(sT^{n-1})} (e_n) \neq 0$, we get
\begin{align*}
\lambda \vee (1-\lambda) = 1
\end{align*}
for any $0<\lambda <1$. This is a contradiction.
Hence, $h_{Z(sT^{n-1})} (e_n)= 0$ for any $s > 0$.
\end{proof}

The following lemma establishes a homogeneity property.
\begin{lem}\label{lem3.2}
Let $n \geq 3$. If $Z : \mathcal{P}_o ^n \to \mathcal {K}_o ^n$ is an $\sln$ contravariant $L_\infty$ Minkowski valuation, then
\begin{align}\label{101}
h_{Z(sT^{n})} (\pm e_i) = sh_{ZT^{n}} (\pm e_i),~~1 \leq i \leq n,
\end{align}
for any $s >0$.
\end{lem}
\begin{proof}
Since $Z$ is $\sln$ contravariant, we only need to show that (\ref{101}) holds for $i=n$.

Define $H_\lambda$ and $\phi_3,\phi_4 \in \sln$ as in Section \ref{s2}. Since $Z$ is a valuation,
\begin{align*}
h_{Z(sT^{n})} (x) \vee h_{Z(sT^{n} \cap H_\lambda)} (x) = h_{Z(sT^{n}\cap H_\lambda ^-)} (x) \vee h_{Z(sT^{n} \cap H_\lambda ^+)} (x),
\end{align*}
for any $x \in \mathbb{R}^n$, $s>0$.
By Lemma \ref{lem3.1}, $h_{Z(sT^{n} \cap H_\lambda)} (x) = 0$. Thus,
\begin{align*}
h_{Z(sT^{n})} (x)= h_{Z(sT^{n}\cap H_\lambda ^-)} (x) \vee h_{Z(sT^{n} \cap H_\lambda ^+)} (x).
\end{align*}
Note that $sT^{n}\cap H_\lambda ^- = \phi_3 \lambda^{1/n} sT^{n}$, $sT^{n}\cap H_\lambda ^+ = \phi_4 (1-\lambda)^{1/n}sT^{n}$. Since $Z$ is $\sln$ \text{contravariant}, we have
\begin{align}\label{a7}
h_{Z(sT^{n})} (x)&= h_{Z(\phi_3 \lambda^{1/n} sT^{n})} (x) \vee h_{Z(\phi_4 (1-\lambda)^{1/n}sT^{n})} (x) \nonumber\\
&= h_{Z(\lambda^{1/n} sT^{n})} (\phi_3 ^{-1} x) \vee h_{Z((1-\lambda)^{1/n}sT^{n})} (\phi_4^{-1}x),
\end{align}
where $x=(x_1, \dots, x_n)^t$, $\phi_3 ^{-1} x = \lambda^{1/n} (\frac{1}{\lambda} x_1, \frac{\lambda -1 }{\lambda}x_1 +x_2, x_3,\dots,x_n)^t$ and $\phi_4 ^{-1} x = (1-\lambda)^{1/n}(x_1 - \frac{\lambda}{1-\lambda} x_2, \frac{1}{1-\lambda}x_2,x_3,\dots,x_n)^t$.
If we choose $x = e_n$, then
\begin{align*}
h_{Z(sT^{n})} (e_n)
=h_{\lambda ^{1/n} Z(\lambda ^{1/n} sT^{n})} (e_n) \vee h_{(1- \lambda)^{1/n}Z((1-\lambda)^{1/n}sT^{n})} (e_n)
\end{align*}
for any $0 < \lambda <1$ and $s>0$.
Taking $\lambda = \frac{\lambda_1}{\lambda_2}$, $0 < \lambda_1 < \lambda_2$ and then taking $s = \lambda_2 ^{1/n}$, with (\ref{210}), we get
\begin{align}\label{a1}
h_{\lambda_2^{1/n}Z(\lambda_2^{1/n}T^{n})} (e_n)
= h_{\lambda_1^{1/n}Z(\lambda_1^{1/n}T^{n})} (e_n) \vee h_{(\lambda_2-\lambda_1)^{1/n}Z((\lambda_2-\lambda_1)^{1/n}T^{n})} (e_n)
\end{align}
for any $0 < \lambda_1 < \lambda_2$.

Let $f(\lambda) = h_{\lambda^{1/n}Z(\lambda^{1/n}T^{n})} (e_n)$, $\lambda >0$. Hence $f$ satisfies the condition in Lemma \ref{cauchy}.
Thus we have $$h_{\lambda^{1/n}Z(\lambda^{1/n}T^{n})} (e_n) = h_{ZT^{n}} (e_n).$$
This shows $h_{Z(sT^{n})} (e_n) = sh_{ZT^{n}} (e_n)$ for any $s>0$. Similarly, $h_{Z(sT^{n})} (-e_n) = sh_{ZT^{n}} (-e_n)$ for any $s>0$.

\end{proof}

\begin{proof}[\textbf{Proof of Theorem \ref{thm1.1}}]
In Section \ref{s2}, we have already shown that $\hat{\Pi}_\infty ^+$ and $\hat{\Pi}_\infty ^-$ are $\sln$ contravariant $L_\infty$ Minkowski valuations. Hence $c_1 \hat{\Pi}_\infty ^+ P +_\infty c_2 \hat{\Pi}_\infty ^- P$ is an $\sln$ contravariant $L_\infty$ Minkowski valuation.

Now we need to show that if $Z : \mathcal{P}_o ^n \to \mathcal {K}_o ^n$ is an $\sln$ contravariant $L_\infty$ Minkowski valuation, then there exists constants $c_1, c_2 \geq 0$ such that
\begin{align}\label{27a}
ZP = c_1 \hat{\Pi}_\infty ^+ P +_\infty c_2 \hat{\Pi}_\infty ^- P
\end{align}
for any $P \in \mathcal{P}_o ^n$.

Let $c_1 = h_{Z(sT^n)} (e_1)$ and $c_2 = h_{Z(sT^n)} (-e_1)$. We first want to show that
$$Z(sT^n) = [-c_2s(e_1+\dots +e_n),c_1s(e_1+\dots +e_n)] = c_1 \hat{\Pi}_\infty ^+ (sT^n) +_\infty c_2 \hat{\Pi}_\infty ^- (sT^n)$$
for any $s >0$. The second equality follows directly from the definitions of $\hat{\Pi}_\infty ^+$ and $\hat{\Pi}_\infty ^-$.

We will show that the orthogonal projection of $Z(sT^{n})$ onto any plane spanned by $\{e_i,e_j\}$, $1 \leq i < j \leq n$ is the segment $[-c_2s(e_i+e_j),c_1s(e_i+e_j)]$. By the $\sln$ contravariance of $Z$, we only need to show that $Z(sT^{n})|\mathbb{R}^2$ has the desired result. Since
\begin{align*}
h_{Z(sT^{n})} (x|\mathbb{R}^2) = h_{(Z(sT^{n}))|\mathbb{R}^2} (x),
\end{align*}
we only need to consider $h_{Z(sT^{n})} (\alpha e_1 + \beta e_2)$. Also since the support function is continuous, we will further assume that $\alpha, \beta$ are not zero.

If $\alpha, \beta$ have the same sign, taking $x=\alpha e_1 +\beta e_2$, $\lambda = \frac{\alpha}{\alpha+\beta}$ in (\ref{a7}), with (\ref{210}), we obtain that
\begin{align*}
h_{Z(sT^{n})} (\alpha e_1 +\beta e_2) = h_{\lambda^{1/n}Z(\lambda^{1/n}sT^{n})} ((\alpha+\beta)e_1) \vee h_{(1-\lambda)^{1/n}Z((1-\lambda)^{1/n}sT^{n})} ((\alpha+\beta)e_2).
\end{align*}
Combined with the Lemma \ref{lem3.2}, we get
\begin{align*}
h_{Z(sT^{n})} (\alpha e_1 +\beta e_2)
&=h_{Z(sT^{n})} ((\alpha+\beta)e_1) \vee h_{Z(sT^{n})}((\alpha+\beta)e_2).
\end{align*}
If $\alpha,\beta > 0$, we get
$$h_{Z(sT^{n})} (\alpha e_1 +\beta e_2)= c_1 s(\alpha+\beta)= h_{[-c_2s(e_1+e_2),c_1s(e_1+e_2)]}(\alpha e_1 +\beta e_2).$$
If $\alpha,\beta <0$, we get
$$h_{Z(sT^{n})} (\alpha e_1 +\beta e_2)= -c_2 s(\alpha+\beta)= h_{[-c_2s(e_1+e_2),c_1s(e_1+e_2)]}(\alpha e_1 +\beta e_2).$$

If $\alpha > -\beta >0$ or $-\alpha > \beta >0$, taking $x= (\alpha +\beta) e_1$, $\lambda = \frac{\alpha+\beta}{\alpha}$, $s=\lambda ^{-1/n}s$ in (\ref{a7}), with (\ref{210}), we obtain
\begin{align*}
h_{\lambda^{-1/n}Z(\lambda^{-1/n}sT^{n})} ((\alpha +\beta) e_1)
=h_{Z(sT^{n})} (\alpha e_1+\beta e_2) \vee h_{(\frac{1}{\lambda}-1)^{1/n}Z((\frac{1}{\lambda}-1)^{1/n}sT^{n})} ((\alpha+\beta)e_1).
\end{align*}
Combined with Lemma \ref{lem3.2}, we get
\begin{align*}
h_{Z(sT^{n})} (\alpha e_1+\beta e_2) &\leq h_{\lambda^{-1/n}Z(\lambda^{-1/n}sT^{n})} ((\alpha +\beta) e_1) \\
&=h_{Z(sT^{n})}((\alpha +\beta) e_1) \\
&=h_{[-c_2s(e_1+e_2),c_1s(e_1+e_2)]}(\alpha e_1 +\beta e_2).
\end{align*}

If $\beta > -\alpha >0$ or $-\beta > \alpha >0$, taking $x= (\alpha +\beta) e_2$, $\lambda = -\frac{\alpha}{\beta}$, $s=(1-\lambda)^{-1/n} s$ in (\ref{a7}), we obtain
\begin{align*}
&h_{(1-\lambda) ^{-1/n}Z((1-\lambda) ^{-1/n}sT^{n})} ((\alpha +\beta) e_2) \\
&=h_{(1-\lambda) ^{-1/n} \lambda^{1/n}Z((1-\lambda) ^{-1/n} \lambda^{1/n}sT^{n})} ((\alpha +\beta) e_2) \vee h_{Z(sT^{n})} (\alpha e_1+\beta e_2).
\end{align*}
Similarly, we get
$$h_{Z(sT^{n})} (\alpha e_1+\beta e_2) \leq h_{[-c_2s(e_1+e_2),c_1s(e_1+e_2)]}(\alpha e_1 +\beta e_2).$$

Combined, we get
$$h_{(Z(sT^{n}))|\mathbb{R}^2} (x) = h_{Z(sT^{n})} (x) \leq h_{[-c_2s(e_1+e_2),c_1s(e_1+e_2)]}(x)$$
for an arbitrary $x \in \mathbb{R}^2$ by the continuity of the support function.
Hence we get that $(Z(sT^{n}))|\mathbb{R}^2 \subset [-c_2s(e_1+e_2),c_1s(e_1+e_2)]$. Since $(Z(sT^{n}))|\mathbb{R}^2$ is convex, there exist real $a,b$ with $-c_2 \leq a \leq b \leq c_1$ such that $(Z(sT^{n}))|\mathbb{R}^2 = [as(e_1+e_2),bs(e_1+e_2)]$. However, $h_{(Z(sT^{n}))|\mathbb{R}^2}(e_1) = h_{Z(sT^{n})} (e_1)=c_1s$ and $h_{(Z(sT^{n}))|\mathbb{R}^2}(-e_1) = h_{Z(sT^{n})} (-e_1)= c_2s$ show that $a=-c_2$, $b=c_1$. Hence, $(Z(sT^{n}))|\mathbb{R}^2 = [-c_2s(e_1+e_2),c_1s(e_1+e_2)]$.

Since the orthogonal projection of $Z(sT^{n})$ onto any plane spanned by $\{e_i,e_j\}$, $1 \leq i < j \leq n$ is the segment $[-c_2s(e_i+e_j),c_1s(e_i+e_j)]$, we obtain that $Z(sT^{n}) = [-c_2s(e_1+\dots +e_n),c_1s(e_1+\dots +e_n)]$.

By the $\sln$ contravariance of $Z$, (\ref{27a}) holds true for every simplex in $\mathcal {T}_o^n$.
Assume that (\ref{27a}) holds on $\mathcal{P}_{i-1}$, $i \geq 2$. Let $P=P_1 \cup P_2 \in \mathcal{P}_{i}$, where $P_1,P_2 \in \mathcal{P}_{i-1}$ have disjoint relative interiors. We can assume $P \neq P_1$ and $P \neq P_2$. Set $d= \dim P_1 = \dim P_2$, $\dim (P_1 \cap P_2) = d-1$. By (\ref{303}), we have $P_1 \cap P_2 \in \mathcal{P}_{i-1}$. Hence,
$$h_{Z(P_1 \cap P_2)} = 0 \leq h_{ZP_i}$$
for $i=1,2$. Therefore $Z(P_1 \cup P_2)$ is uniquely determined by $h_{Z(P_1 \cup P_2)} = h_{ZP_1} \vee h_{ZP_2}$.
Thus (\ref{27a}) holds on $\mathcal{P}_{i}$. For any $P \in \mathcal{P}_o^n$, there exists $i$ such that $P \in \mathcal{P}_{i}$. Thus (\ref{27a}) holds on $\mathcal{P}_o^n$.
\end{proof}


\section{$\sln$ covariant $L_\infty$ Minkowski valuations}
We will use the following lemma by Ludwig \cite{Lud05} and Haberl \cite{Hab12b} for maps to $\mathcal{K}_o ^n$ and Parapatits \cite{Par14b} for maps to $C_p(\mathbb{R}^n)$, the set of $p$-homogenous continuous functions on $\mathbb{R}^n$. (Maps to $C_p(\mathbb{R}^n)$ are considered in Section \ref{Slp}.)

\begin{lem}\label{lem4.1}
Let $n \geq 2$. If a map $\Phi : \mathcal{P}_o ^n \to C_p(\mathbb{R}^n)$ is $\sln$ covariant, then
\begin{align*}
  \Phi(P) (x) = \Phi(P) (x|P), ~x \in \mathbb{R}^n
\end{align*}
for any $P \in \mathcal{P}_o ^n$. In particular, if a map $Z : \mathcal{P}_o ^n \to \mathcal{K}_o ^n$ is $\sln$ covariant (hence $P \mapsto h_{ZP}$ is also $\sln$ covariant), then $ZP \subset \lin P$, and
\begin{align*}
  h_{ZP} (x) = h_{ZP} (x|P), ~x \in \mathbb{R}^n
\end{align*}
for any $P \in \mathcal{P}_o ^n$.
\end{lem}

The following Lemma determines the constants in Theorem \ref{thm1.3} and establishes a homogeneity property of $\sln$ covariant $L_\infty$ Minkowski valuations.
\begin{lem}\label{lem4.2}
Let $n \geq 3$. If $Z : \mathcal{P}_o ^n \to \mathcal {K}_o ^n$ is an $\sln$ covariant $L_\infty$ Minkowski valuation, then
\begin{align}\label{d1}
h_{Z(sT^d)} ( \pm e_1) = sh_{ZT^d} ( \pm e_1)
\end{align}
for $1 \leq d \leq n$ and $s>0$, while
\begin{align}\label{d2}
h_{ZT^1} ( \pm e_1) \leq \dots \leq h_{ZT^n} ( \pm e_1).
\end{align}
\end{lem}
\begin{proof}
Let $a_d : = h_{ZT^d} (e_1)$, $b_d : = h_{ZT^d} (-e_1)$ for $1 \leq d \leq n$.

If $d \leq n-1$, it is easy to see that $Z(sT^{d}) = s ZT^d$ by the $\sln$ covariance of $Z$. Hence (\ref{d1}) holds for $d \leq n-1$.

For $0 < \lambda < 1$, define $H_\lambda$, $\phi_1,~\phi_2,~\phi_3$, and $\phi_4$ as in Section \ref{s2}. Since $Z$ is an $L_\infty$ Minkowski valuation,
\begin{align}\label{21}
h_{Z(sT^{d})} (x) \vee h_{Z(sT^d \cap H_\lambda)} (x) = h_{Z(sT^d \cap H_\lambda ^-)} (x) \vee h_{Z(sT^d \cap H_\lambda ^+)} (x),~~x \in \mathbb{R}^n
\end{align}
for any $s >0$.

For $2 \leq d \leq n-1$, since $Z$ is $\sln$ covariant, we obtain
\begin{align}\label{17}
h_{ZT^d} (x) \vee h_{Z\hat{T}^{d-1}} (\phi _1 ^t x)
= h_{ZT^d} (\phi _1 ^t x) \vee h_{Z T^d} (\phi _2 ^t x),
\end{align}
where $x=(x_1, \dots, x_n)^t \in \mathbb{R}^n$, $\phi_1 ^t x = (\lambda x_1 + (1-\lambda)x_2, x_2, x_3,\dots,x_{n-1}, \frac{1}{\lambda} x_n)^t$ and $\phi_2 ^t x = (x_1, \lambda x_1 + (1-\lambda)x_2,x_3,\dots,x_{n-1},\frac{1}{\lambda}x_n)^t$.
Taking $x=e_1$, $s=1$ in (\ref{17}), we get
\begin{align*}
h_{ZT^d} (e_1) \vee h_{Z\hat{T}^{d-1}} (\lambda e_1)
= h_{ZT^d} (\lambda e_1) \vee h_{Z T^d} (e_1 + \lambda e_2).
\end{align*}
Also since support functions are homogeneous and continuous, and $h_{Z\hat{T}^{d-1}} (e_1) = a_{d-1}$ by the $\sln$ covariance of $Z$,
\begin{align}\label{18}
a_d \vee (\lambda a_{d-1})
= (\lambda a_d) \vee h_{Z T^d} (e_1 + \lambda e_2)
\end{align}
holds for $0 \leq \lambda \leq 1$.

We need to show that $a_{d-1} \leq a_d$. Indeed, if we assume $a_d < a_{d-1}$, then there exists $0 \leq \lambda_0 <1$ such that $a_{d-1} \lambda_0 = a_d$. Taking $\lambda_0 \leq \lambda \leq 1 $ in (\ref{18}), we get $h_{Z T^d} (e_1 + \lambda e_2)  = a_{d-1} \lambda$. However, choosing $\lambda_0 \leq \lambda_1 < \lambda_2 \leq 1$, by the sublinearity of the support function, we have
\begin{align*}
  a_{d-1} \lambda_2 =  h_{Z T^d} (e_1 + \lambda_2 e_2) &\leq h_{Z T^d} (e_1 + \lambda_1 e_2) + h_{Z T^d} ((\lambda_2 -\lambda_1) e_2) \\
   &= a_{d-1} \lambda_1 + a_d(\lambda_2 -\lambda_1),
\end{align*}
which is a contradiction to the assumption.

Similarly, taking $x=-e_1$ in (\ref{17}), we get $b_{d-1} \leq b_d$.

\vskip 5pt
If $d=n$, define $\phi_3,\phi_4 \in \sln$ as in Section \ref{s2}. Since $Z$ is $\sln$ covariant, (\ref{21}) shows that
\begin{align}\label{30}
h_{Z(sT^{n})} (x) \vee h_{Z(\lambda^{1/n}s\hat{T}^{n-1})} (\phi_3 ^t x)
=h_{Z(\lambda^{1/n} sT^{n})} (\phi_3 ^t x) \vee h_{Z((1-\lambda)^{1/n}sT^{n})} (\phi_4 ^t x),
\end{align}
where $x=(x_1, \dots, x_n)^t$, $\phi_3 ^t x = \lambda^{-1/n} (\lambda x_1 + (1-\lambda)x_2, x_2, x_3,\dots, x_n)^t$ and $\phi_4 ^t x = (1-\lambda)^{-1/n}(x_1, \lambda x_1 + (1-\lambda)x_2,x_3,\dots,x_n)^t$.
So if we choose $x =e_n$ in (\ref{30}), we have
\begin{align}\label{31}
h_{Z(sT^{n})} (e_n) \vee h_{\lambda^{-1/n}Z(\lambda^{1/n}s\hat{T}^{n-1})} (e_n)
= h_{\lambda^{-1/n}Z(\lambda^{1/n} sT^{n})} (e_n) \vee h_{(1-\lambda)^{-1/n}Z((1-\lambda)^{1/n}sT^{n})} (e_n)
\end{align}
for $0 < \lambda <1$, $s>0$. Since (\ref{d1}) holds for $d \leq n-1$ and $Z$ is $\sln$ covariant, we have $h_{\lambda^{-1/n}Z(\lambda^{1/n}s\hat{T}^{n-1})} (e_n) = a_{n-1}$.
Combining it with (\ref{210}), taking $\lambda = \frac{\lambda_1}{\lambda_2}$, $0 < \lambda_1 < \lambda_2$ and $s = \lambda_2 ^{1/n}$ in (\ref{31}), we get
\begin{align}\label{32}
h_{\lambda_2^{-1/n}Z(\lambda_2^{1/n}T^{n})} (e_n) \vee a_{n-1}
= h_{\lambda_1^{-1/n}Z(\lambda_1^{1/n}T^{n})} (e_n) \vee h_{(\lambda_2-\lambda_1)^{-1/n}Z((\lambda_2-\lambda_1)^{1/n}T^{n})} (e_n)
\end{align}
for $0 < \lambda_1 < \lambda_2$.

Let $f(\lambda) = h_{\lambda^{-1/n}Z(\lambda^{1/n}T^{n})} (e_n)$, $\lambda >0$. Hence $f$ satisfies the condition in Lemma \ref{cauchy}.
Thus we have $h_{\lambda^{-1/n}Z(\lambda^{1/n}T^{n})} (e_n) = h_{ZT^{n}} (e_n) \geq a_{n-1}$. Combined with the $\sln$ covariance of $Z$, we have
$$h_{Z(sT^{n})} (e_1) = sh_{ZT^{n}} (e_1),  h_{ZT^{n}} (e_1) \geq a_{n-1} = h_{ZT^{n-1}} (e_1).$$

Similarly, taking $x=-e_1$ in (\ref{30}), we get
$$h_{Z(sT^{n})} (-e_1) = sh_{ZT^{n}} (-e_1),  h_{ZT^{n}} (-e_1) \geq  h_{ZT^{n-1}} (-e_1).$$
\end{proof}

\begin{proof}[\textbf{Proof of Theorem \ref{thm1.3}}]
It is easy to see that the identity map and the reflection map are $\sln$ covariant $L_\infty$ Minkowski valuations. Hence
\begin{align}\label{27}
ZP = a_d P +_\infty (- b_d P) = [a_d P , - b_d P]
\end{align}
is also an $\sln$ covariant $L_\infty$ Minkowski valuation.

Now we will show that if $Z$ is an $\sln$ covariant $L_\infty$ Minkowski valuation, then (\ref{27}) holds.
We will first show that (\ref{27}) holds for simplices $sT^d$, $d \leq n$, $s > 0$. We will prove the result by induction on the dimension $d$. $Z \{ o \} = \{ o \}$ has been shown in (\ref{37}). Set $a_d : = h_{ZT^d} (e_1)$ and $b_d : = h_{ZT^d} (-e_1)$. Lemma \ref{lem4.2} shows that
$$0 \leq a_1 \leq \dots \leq a_n,~~0 \leq b_1 \leq \dots \leq b_n.$$

If $d=1$, by the $\sln$ covariance of $Z$, we have $Z[0, se_1] = sZ[0, e_1]$ for any $s > 0$. By Lemma \ref{lem4.1}, we get that $Z[0, e_1] = [-b_1,a_1]$. The case $d=1$ is done.

Assume that (\ref{27}) holds true for dimension $d-1$, $2 \leq d \leq n$. We want to show that (\ref{27}) also holds true for dimension $d$.

We will show by induction on the number $m$ of coordinates of $x$ not equal to zero that
\begin{align}\label{26}
  h_{Z (sT^d)} (x) = h_{[a_d sT^d , -b_d sT^d]} (x).
\end{align}

For $m=1$, (\ref{26}) holds true by (\ref{d1}), the $\sln$ covariance of $Z$ and the homogeneity of the support function. Assume that $(\ref{26})$ holds true for $m-1$. We need to show that $(\ref{26})$ also holds true for $m$. By the $\sln$ covariance of $Z$, we can assume w.l.o.g. that $x=x_1 e_1+\dots +x_m e_m$, $x_1,\dots,x_m \neq 0$.

Note that (\ref{17}) is a special form of (\ref{30}) for dimension $d \leq n-1$ since $Z(sT^d) = sZT^d$ for any $s>0$. We will use (\ref{30}) to get the value of $h_{ZT^d}$ not just for $d=n$ but also for $d \leq n-1$.

For $x_1 > x_2 > 0$ or $0 > x_2 >x_1$, taking $x=x_1 e_1 + x_3 e_3 + \dots +x_m e_m$, $\lambda = \frac{x_2}{x_1}$, $s=(1-\lambda) ^{-1/d}s$ in (\ref{30}), by (\ref{210}), we get
\begin{align}\label{38}
&h_{(1-\lambda) ^{1/d}Z((1-\lambda) ^{-1/d}sT^d)} (x_1e_1 + x_3 e_3 + \dots +x_m e_m)  \nonumber \\
 & \qquad \qquad \vee h_{(1-\lambda) ^{1/d}\lambda ^{-1/d}Z((1-\lambda) ^{-1/d}\lambda ^{1/d}s\hat{T}^{d-1}} (x_2 e_1 + x_3 e_3 + \dots +x_m e_m) \nonumber \\
= &h_{(1-\lambda) ^{1/d}\lambda ^{-1/d}Z((1-\lambda) ^{-1/d}\lambda ^{1/d}sT^d)} (x_2 e_1 + x_3 e_3 + \dots +x_m e_m)   \nonumber \\
 & \qquad \qquad \vee h_{Z(s T^d)} (x_1 e_1 + \dots +x_m e_m).
\end{align}

Since $a_{d-1} \leq a_d$, $b_{d-1} \leq b_d$, $|x_2| < |x_1|$, combining the induction assumption with the $\sln$ covariance of $Z$, we have
\begin{align*}
&h_{(1-\lambda) ^{1/d}Z((1-\lambda) ^{-1/d}sT^d)} (x_1e_1 + x_3 e_3 + \dots +x_m e_m)  \nonumber\\
& \qquad \qquad = \max \{ a_ds x_i, -b_ds x_i : 1 \leq i \leq m~\text{and}~i \neq 2\}  \nonumber\\
& \qquad \qquad \geq \max \{ a_{d-1} s x_i, -b_{d-1} s x_i : 2 \leq i \leq m\}  \nonumber\\
& \qquad \qquad = h_{(1-\lambda) ^{1/d}\lambda ^{-1/d}Z((1-\lambda) ^{-1/d}\lambda ^{1/d}s\hat{T}^{d-1}} (x_2 e_1 + x_3 e_3 + \dots +x_m e_m).
\end{align*}
It follows from (\ref{38}) that
\begin{align}\label{49}
&h_{Z(sT^d)} (x_1 e_1 + \dots +x_m e_m) \nonumber \\
& \qquad \qquad \leq h_{(1-\lambda) ^{1/d}Z((1-\lambda) ^{-1/d}sT^d)} (x_1e_1 + x_3 e_3 + \dots +x_m e_m) \nonumber \\
& \qquad \qquad = \max \{ a_ds x_i, -b_ds x_i : 1 \leq i \leq m\}.
\end{align}

For $x_2 > x_1 > 0$ or $0 > x_1 >x_2$,
taking $x=x_2 e_2 + x_3 e_3 + \dots +x_m e_m$, $1-\lambda = \frac{x_1}{x_2}$, $s=\lambda ^{-1/d}s$ in (\ref{30}), by (\ref{210}), we get
\begin{align}\label{40}
&h_{\lambda ^{1/d}Z(\lambda ^{-1/d}sT^d)} (x_2e_2 + x_3 e_3 + \dots +x_m e_m) \nonumber \\
& \qquad \qquad \vee h_{Z(s\hat{T}^{d-1})} (x_1 e_1 + \dots +x_m e_m) \nonumber \\
= &h_{Z(sT^d)} (x_1 e_1 + \dots +x_m e_m) \nonumber \\
 & \qquad \qquad \vee h_{(1-\lambda) ^{-1/d}\lambda ^{1/d}Z((1-\lambda) ^{1/d}\lambda ^{-1/d}s T^d)} (x_1 e_2 +x_3 e_3 + \dots +x_m e_m).
\end{align}
Similarly to the case $|x_2| < |x_1|$, since
$$h_{\lambda ^{1/d}Z(\lambda ^{-1/d}sT^d)} (x_2e_2 + x_3 e_3 + \dots +x_m e_m) \geq h_{Z(s\hat{T}^{d-1})} (x_1 e_1 + \dots +x_m e_m),$$
we get
\begin{align}\label{50}
&h_{Z(sT^d)} (x_1 e_1 + \dots +x_m e_m) \nonumber \\
& \qquad \qquad \leq h_{\lambda ^{1/d}Z(\lambda ^{-1/d}sT^d)} (x_2e_1 + x_3 e_3 + \dots +x_m e_m) \nonumber \\
& \qquad \qquad = \max \{ a_ds x_i, -b_ds x_i : 1 \leq i \leq m\}.
\end{align}

For $x_1 >0 > x_2$ or $x_2 > 0 > x_1$, taking $0 < \lambda = \frac{x_2}{x_2 - x_1} <1$ and $x= x_1 e_1 + \dots +x_m e_m$ in (\ref{30}), we get
\begin{align}\label{23}
&h_{Z(sT^d)} (x_1 e_1 + \dots +x_m e_m) \nonumber \\
& \qquad \qquad \vee h_{\lambda ^{-1/d}Z(\lambda ^{1/d}s\hat{T}^{d-1})} (x_2 e_2 + x_3 e_3 + \dots +x_m e_m) \nonumber \\
= & h_{\lambda ^{-1/d}Z(\lambda ^{1/d}sT^d)} (x_2 e_2 + x_3 e_3 + \dots +x_m e_m) \nonumber \\
& \qquad \qquad \vee h_{(1-\lambda) ^{-1/d}Z((1-\lambda) ^{1/d}sT^d)} (x_1 e_1 + x_3 e_3 + \dots +x_m e_m).
\end{align}
Combined with the induction assumption and the $\sln$ covariance of $Z$, we have
\begin{align}\label{24}
h_{Z (sT^d)} (x_1 e_1 + \dots +x_m e_m) \leq \max \{ a_d sx_i, -b_d sx_i : 1 \leq i \leq m \}.
\end{align}
Combining (\ref{49}), (\ref{50}) and (\ref{24}) with the continuity of the support function, we get
\begin{align*}
  h_{Z(sT^d)|\mathbb{R}^m} (x_1e_1 + \dots +x_m e_m) &= h_{Z (sT^d)} (x_1e_1 + \dots +x_m e_m) \\
  &\leq h_{[a_d sT^d , -b_d sT^d]} (x_1e_1 + \dots +x_m e_m) \\
  &= h_{[a_d sT^m ,-b_d sT^m]} (x_1e_1 + \dots +x_m e_m)
\end{align*}
for any $x_1,\dots,x_m \in \mathbb{R}$. Thus, $Z(sT^d)|\mathbb{R}^m \subset [a_d sT^m , -b_d sT^m]$.
For any $y \in [a_d sT^m , -b_d sT^m]$ with $y \neq a_d se_1$, we have $y \cdot e_1 < a_d s$, and also $h_{Z(sT^d)|\mathbb{R}^m}(e_1) =h_{Z (sT^d)}(e_1) = a_d s$.
Thus, we obtain $a_d se_1 \in Z(sT^d)|\mathbb{R}^m$. Similarly, $a_d se_i, -b_d se_i \in Z(sT^d)|\mathbb{R}^m$, $1 \leq i \leq m$. Hence, we have
\begin{align*}
  [a_d sT^m , -b_d sT^m] &= s[a_d e_1,\dots,a_d e_m,-b_d e_1,\dots,-b_d e_m] \\
  &\subset Z(sT^d)|\mathbb{R}^m \subset [a_d sT^m , -b_d sT^m].
\end{align*}
That means
\begin{align*}
h_{Z (sT^d)} (x_1e_1 + \dots +x_m e_m) = h_{[a_d sT^d , -b_d sT^d]} (x_1e_1 + \dots +x_m e_m)
\end{align*}
for any $x_1,\dots,x_m \in \mathbb{R}$. The induction is complete.

By the $\sln$ covariance of $Z$, (\ref{27}) holds true for any simplex in $\mathcal {T}_o^n$.
Assume that (\ref{27}) holds on $\mathcal{P}_{i-1}$, $i \geq 2$. Let $P=P_1 \cup P_2 \in \mathcal{P}_{i}$, where $P_1,P_2 \in \mathcal{P}_{i-1}$ have disjoint relative interiors. We can assume $P \neq P_1$ and $P \neq P_2$. Set $d= \dim P_1 = \dim P_2$, $\dim (P_1 \cap P_2) = d-1$. By (\ref{303}), we have $P_1 \cap P_2 \in \mathcal{P}_{i-1}$. Hence,
$$h_{Z(P_1 \cap P_2)} = h_{[a_{d-1} (P_1 \cap P_2) , -b_{d-1} (P_1 \cap P_2)]} \leq h_{[a_{d-1} P_i, -b_{d-1} P_i]} \leq h_{[a_d P_i, -b_d P_i]} = h_{ZP_i}$$
for $i=1,2$. Therefore
$$h_{Z(P_1 \cup P_2)} = h_{ZP_1} \vee h_{ZP_2} = h_{[a_{d} (P_1 \cup P_2) , -b_{d} (P_1 \cup P_2)]}.$$
Thus (\ref{27}) holds on $\mathcal{P}_{i}$. For any $P \in \mathcal{P}_o^n$, there exists $i$ such that $P \in \mathcal{P}_{i}$. Thus (\ref{27}) holds on $\mathcal{P}_o^n$.
\end{proof}


\section{$\sln$ covariant $L_p$ Minkowski valuations and function-valued valuations}\label{Slp}
First, let us consider function-valued valuations as Parapatits did in \cite{Par14a,Par14b}.
Let $1 \leq p < \infty$ throughout this section if there are no further remarks. The function $f: \mathbb{R}^n \to \mathbb{R}$ is \emph{$p$-homogenous} if
\begin{align*}
f(\lambda x) = \lambda ^p f(x), ~x \in \mathbb{R}^n
\end{align*}
for any $\lambda \geq 0$. Let $C_p(\mathbb{R}^n)$ be the set of $p$-homogenous continuous functions on $\mathbb{R}^n$. We call $\Phi : \mathcal{P}_o ^n \to C_p(\mathbb{R}^n)$ a \emph{valuation} if
\begin{align*}
\Phi (K \cup L) + \Phi (K \cap L) = \Phi (K) + \Phi (L)
\end{align*}
whenever $K \cup L, K \cap L, K,L \in \mathcal{P}_o ^n$. Here the addition is the ordinary addition of functions.

We call $\Phi : \mathcal{P}_o ^n \to C_p(\mathbb{R}^n)$ is \emph{$\sln$ \text{\rm(or} $\gln$\text{\rm)} covariant} if
$$\Phi (\phi K) (x) = \Phi (K) (\phi ^t x)$$
for any $K \in \mathcal{P}_o ^n$ and any $\phi \in \sln~(\text{or}~\gln)$.

The map $Z: \mathcal{P}_o ^n \to \mathcal{K}_o ^n$ is an $\sln$ (or $\gln$) covariant $L_p$ Minkowski valuation if and only if $\Phi : P \mapsto h_{ZP}^p$ is an $\sln$ (or $\gln$) covariant valuation.


\begin{lem}[Haberl \cite{Hab12b} and Parapatits \cite{Par14b}]\label{lpf}
Let $n \geq 3$ and $\Phi$ map $\mathcal{P}_o ^n$ to $C_p(\mathbb{R}^n)$. \text{Assume} further that, for every $y \in \mathbb{R}^n$, the function $s \mapsto \Phi(sT^n)(y)$ is bounded from below on some non-empty open interval $I_y \subset (0,+\infty)$. Also assume that $\Phi$ is continuous at the \text{interval} $[o,e_1]$. Then $\Phi$ is an $\sln$ covariant valuation if and only if there exist constants $c_1,c_2,c_3,c_4 \in \mathbb{R}$ such that
\begin{align*}
\Phi P = c_1 h_{M_p^+ P} ^p + c_2h_{M_p^- P}^p + c_3 h_{P}^p + c_4 h_{-P}^p
\end{align*}
for every $P \in \mathcal {P}_o ^n$.
\end{lem}

In \cite{Hab12b}, Haberl just considered the valuation $P \mapsto h_{ZP}$, where $Z$ is a Minkowski valuation. Hence he has the restrictions that $c_1,c_2,c_3,c_4 \geq 0$. However, his method also can be used to get this Lemma for $p=1$. This also works for Lemma \ref{lem5.7} below.



We remove the assumption that $\Phi$ is continuous at the interval $[o,e_1]$ and get the following result.


\begin{lem}\label{lem5.4}
Let $n \geq 3$ and $\Phi$ map $\mathcal{P}_o ^n$ to $C_p(\mathbb{R}^n)$. Assume further that, for every $y \in \mathbb{R}^n$, the function $s \mapsto \Phi(sT^n)(y)$ is bounded from below on some non-empty open interval $I_y \subset (0,+\infty)$. Then $\Phi$ is an $\sln$ covariant valuation if and only if there exist constants $a_1,a_2,b_1,b_2,c_1,c_2 \in \mathbb{R}$ such that
\begin{align*}
\Phi P = c_1 h_{M_p^+ P} ^p + c_2h_{M_p^- P}^p + \Phi_{p;a_1,a_2}P + \Phi_{p;b_1,b_2}(-P)
\end{align*}
for every $P \in \mathcal {P}_o ^n$, where $\Phi_{p;a_1,a_2}$ is defined as follows.
\end{lem}

For $1 \leq j \leq \dim P-1$, let $\mathcal{F}_{j,o} (P)$ denote the set of $j$-dimensional faces of $P \in \mathcal {P}_o^n$ that contain the origin.
Let $a_1,a_2 \in \mathbb{R}$. For $P \in \mathcal {P}_o ^n$, define $\Phi_{p;a_1,a_2}(P)$ by
\begin{align*}
\Phi_{p;a_1,a_2}P = a_1 h_P^p
+ (a_2-a_1) \sum_{1 \leq j \leq \dim P -1} (-1)^j \sum_{F \in \mathcal{F}_{j,o} (P)}h_{F}^p
\end{align*}
if $\dim P$ is odd; and
\begin{align*}
\Phi_{p;a_1,a_2}P = (2a_2 -a_1) h_P^p
+ (a_2-a_1) \sum_{1 \leq j \leq \dim P -1} (-1)^j \sum_{F \in \mathcal{F}_{j,o} (P)} h_{F}^p
\end{align*}
if $\dim P $ is even.

For $1 < p < \infty$, $n \geq 3$ and $p=1$, $n \geq 4$,
if we further assume that $\Phi P$ is non-negative and $(\Phi P)^{1/p}$ is sublinear for every $P \in \mathcal{P}_o ^n$,
then we obtain Theorem \ref{thm5.1} which is equivalent to Theorem \ref{thm1.4} and Theorem \ref{thm1.6}.

\begin{thm}\label{thm5.1}
Let $n \geq 3$, $1 < p < \infty$ or $n \geq 4$, $p=1$, and $\Phi$ map $\mathcal{P}_o ^n$ to $C_p(\mathbb{R}^n)$. Assume further that $\Phi P$ is non-negative and $(\Phi P)^{1/p}$ is sublinear for every $P \in \mathcal{P}_o ^n$. Then $\Phi$ is an $\sln$ covariant valuation if and only if there exist constants $a_1,b_1,c_1,c_2 \geq 0$ such that
\begin{align*}
\Phi P = c_1 h_{M_p^+ P} ^p + c_2h_{M_p^- P}^p + a_1 h_{P}^p + b_1 h_{-P}^p
\end{align*}
for every $P \in \mathcal {P}_o ^n$.
\end{thm}

Now we begin to prove Lemma \ref{lem5.4} and Theorem \ref{thm5.1}.

\emph{The inclusion-exclusion principle} states that a function-valued valuation $\Phi$ satisfies
\begin{align*}
\Phi (T_1 \cup \dots \cup T_m) = \sum_i \Phi(T_i) - \sum_{i<j} \Phi(T_i \cap T_j) + \dots
\end{align*}
for any $T_1,\dots,T_m, T_1 \cup \dots \cup T_m \in \mathcal {T}_o^n$. In particular, $\Phi (T_1 \cup \dots \cup T_m)$ does not dependent on the choice of $T_1,\dots,T_m$; see Ludwig and Reitzner \cite{LR2006elementary}.

\begin{proof}[\textbf{Proof of Lemma \ref{lem5.4}}] For $a_1,a_2 \in \mathbb{R}$, we first need to show that $\Phi_{p;a_1,a_2}$ is a valuation.
\begin{lem}\label{lem5.3}
For $a_1,a_2 \in \mathbb{R}$, $\Phi_{p;a_1,a_2}$ is a $\gln$ covariant valuation.
\end{lem}
\begin{proof}
It is easy to see from the definition that $\Phi_{p;a_1,a_2}$ is $\gln$ covariant. Next, we prove that $\Phi_{p;a_1,a_2}$ is a valuation.

Let $K,L \in \mathcal {P}_o ^n$, $K \neq L$. To show that
\begin{align}\label{71}
\Phi_{p;a_1,a_2}(K \cup L) + \Phi_{p;a_1,a_2}(K \cap L) = \Phi_{p;a_1,a_2} (K) + \Phi_{p;a_1,a_2} (L)
\end{align}
whenever $K \cup L$ is convex,
we can assume that $\dim K =\dim L = \dim (K \cup L)$, denoted by $d$. Otherwise (\ref{71}) holds trivially since $K \subset L$ or $L \subset K$.
Hence, we only need to consider the following four cases: \\
(\rmnum{1}) $ o \in \text{relint}\,K \cap \text{relint}\,L $; \\
(\rmnum{2}) $ o \in \text{relint}\,K$, and $o \in \text{relbd}\, L$; \\
(\rmnum{3}) $o \in \text{relbd}\, K \cap \text{relbd}\, L$ and $\dim(K \cap L) = d$; \\
(\rmnum{4}) $o \in \text{relbd}\, K \cap \text{relbd}\, L$ and $\dim(K \cap L) = d - 1$.

First we notice that the map $P \mapsto h_P, P \in \mathcal {P}_o^n$ is a valuation. Hence (\ref{71}) holds true for the case (\rmnum{1}). Also, for case (\rmnum{2}), (\rmnum{3}), we only need to consider the faces containing the origin.

For the case (\rmnum{2}), since $K \cup L$ is convex, we have $\bigcup_{1 \leq j \leq d -1} \mathcal{F}_{j,o} (K \cap L) = \bigcup_{1 \leq j \leq d -1} \mathcal{F}_{j,o}(L)$. Hence (\ref{71}) also holds true.

We will denote the elements of $\mathcal{F}_{j,o} (K)$ by $F_K^j$, and the elements of $\mathcal{F}_{j,o} (L)$ by $F_L^j$.

Now we deal with the case (\rmnum{3}).
For $1 \leq j \leq d-1$, since $K \cup L$ is convex, we can separate $\mathcal{F}_{j,o}(K)$ and $\mathcal{F}_{j,o}(L)$ into five disjoint parts, respectively:
\begin{align}\label{72}
\mathcal{F}_{j,o}(K) &=\mathcal {A}_K^j \cup \mathcal {B}_K^j \cup \mathcal {C}_K^j \cup \mathcal {D}_K^j \cup \mathcal {G}_K^j,
\end{align}
where
\begin{align*}
\mathcal {A}_K^j&= \{ F_K^j : F_K^j \cap \text{relint}\,L \neq \emptyset \}, \nonumber \\
\mathcal {B}_K^j&= \{ F_K^j  : F_K^j \cap L^c \neq \emptyset, \nexists F_L^j~\text{s.t.}~ F_L^j \subset \text{lin} F_K^j \}, \nonumber \\
\mathcal {C}_K^j&= \{ F_K^j  : \exists F_L^j \neq F_K^j,\exists H \in Gr(n,j)~\text{s.t.}~ F_L^j \cup F_K^j \subset H\}, \nonumber \\
\mathcal {D}_K^j&= \{ F_K^j  : \exists F_L^j = F_K^j \},\nonumber \\
\mathcal {G}_K^j&= \{ F_K^j  : \exists F_L^i, i > j ~\text{s.t.}~ F_K^j \subset \text{relint}\,F_L^i\};
\end{align*}
and
\begin{align}\label{73}
\mathcal{F}_{j,o}(L) &=\mathcal {A}_L^j \cup \mathcal {B}_L^j \cup \mathcal {C}_L^j \cup \mathcal {D}_L^j \cup \mathcal {G}_L^j,
\end{align}
where
\begin{align*}
\mathcal {A}_L^j&= \{ F_L^j : F_L^j \cap \text{relint}\,K \neq \emptyset \}, \nonumber \\
\mathcal {B}_L^j&= \{ F_L^j : F_L^j \cap K^c \neq \emptyset, \nexists F_K^j~\text{s.t.}~ F_K^j \subset \text{lin} F_L^j\}, \nonumber\\
\mathcal {C}_L^j&= \{ F_L^j : \exists F_K^j \neq F_L^j,\exists H \in Gr(n,j)~\text{s.t.}~ F_K^j \cup F_L^j \subset H\}, \nonumber \\
\mathcal {D}_L^j&= \{ F_L^j  : \exists F_K^j = F_L^j \},\nonumber \\
\mathcal {G}_L^j&= \{ F_L^j  : \exists F_K^i, i > j ~\text{s.t.}~ F_L^j \subset \text{relint}\,F_K^i\}.
\end{align*}

Set $\mathcal {D}^j := \mathcal {D}_K^j = \mathcal {D}_L^j$.
Since $\text{relbd}\, (K \cup L) = (\text{relbd}\, K \cap L^c) \cup (\text{relbd}\, L \cap K^c) \cup (\text{relbd}\, K \cap \text{relbd}\, L)$,
$\text{relbd}\, (K \cap L) = (\text{relbd}\, K \cap \text{relint}\,L) \cup (\text{relbd}\, L \cap \text{relint}\,K) \cup (\text{relbd}\, K \cap \text{relbd}\, L)$ and $K \cup L$ is convex, we have
\begin{align}\label{74}
&\mathcal{F}_{j,o}(K \cup L) =\mathcal {B}_K^j \cup \mathcal {B}_L^j \cup \mathcal {M}^j \cup (\mathcal {D}^j \cap \mathcal{F}_{j,o}(K \cup L)),
\end{align}
where
\begin{align*}
&\mathcal {M}^j=\{ F_K^j \cup F_L^j: F_K^j \in \mathcal {C}_K^j,F_L^j \in \mathcal {C}_L^j, \exists H \in Gr(n,j), F_K^j \cup F_L^j \subset H \};
\end{align*}
and
\begin{align}\label{75}
&\mathcal{F}_{j,o}(K \cap L) = \mathcal {A}_K^j \cup \mathcal {A}_L^j \cup (\mathcal {N}^j \cap \mathcal{F}_{j,o}(K \cap L)) \cup \mathcal {D}^j \cup \mathcal {G}_K^j \cup \mathcal {G}_L^j,
\end{align}
where
\begin{align*}
&\mathcal {N}^j= \{F_K ^j \cap F_L^j: F_K^j \in \mathcal {C}_K^j,F_L^j \in \mathcal {C}_L^j, \exists H \in Gr(n,j), F_K^j \cup F_L^j \subset H\}.
\end{align*}


Combining (\ref{72}), (\ref{73}), (\ref{74}), (\ref{75}) with the definition of $\Phi _{p;a_1,a_2}$, if
\begin{align}\label{103}
&\sum_{1 \leq j \leq d-1} (-1)^j \sum_{F_K^{j} \in \mathcal{C}_{K}^{j}} h_{F_K^{j}}^p
+ \sum_{1 \leq j \leq d -1} (-1)^j \sum_{F_L^{j} \in \mathcal{C}_{L}^{j}} h_{F_L^{j}}^p
+ 2 \sum_{1 \leq j \leq d -1} (-1)^j \sum_{F \in \mathcal{D}^{j}} h_{F}^p  \nonumber \\
=& \sum_{1 \leq j \leq d-1} (-1)^j \sum_{F_{K\cup L}^{j} \in \mathcal {M}^j} h_{F_{K\cup L}^{j}}^p
+ \sum_{1 \leq j \leq d-1} (-1)^j \sum_{F_{K\cup L}^{j} \in (\mathcal {D}^j \cap \mathcal{F}_{j,o}(K \cup L))} h_{F_{K\cup L}^{j}}^p \nonumber \\
& \qquad + \sum_{1 \leq j \leq d-1} (-1)^j \sum_{F_{K\cap L}^{j} \in (\mathcal {N}^j  \cap \mathcal{F}_{j,o}(K \cup L)) } h_{F_{K\cap L}^{j}}^p + \sum_{1 \leq j \leq d -1} (-1)^j \sum_{F \in \mathcal{D}^{j}} h_{F}^p,
\end{align}
then (\ref{71}) holds true.

Let $F_K^j \in \mathcal{C}_K^j, F_L^j \in \mathcal{C}_L^j$ and $F_K^j \cup F_L^j$ lie in the same $j$-dimensional plane. Since $F_K^j \cup F_L^j$ is convex, $h_{F_K^j \cup F_L^j}^p + h_{F_K^j \cap F_L^j}^p = h_{F_K^j}^p + h_{F_L^j}^p$. Thus
\begin{align}\label{104}
&\sum_{1 \leq j \leq d-1} (-1)^j \sum_{F_K^{j} \in \mathcal{C}_{K}^{j}} h_{F_K^{j}}^p + \sum_{1 \leq j \leq d -1} (-1)^j \sum_{F_L^{j} \in \mathcal{C}_{L}^{j}} h_{F_L^{j}}^p \nonumber \\
& = \sum_{1 \leq j \leq d-1} (-1)^j \sum_{F_{K\cup L}^{j} \in \mathcal {M}^j} h_{F_{K\cup L}^{j}}^p + \sum_{1 \leq j \leq d-1} (-1)^j \sum_{F_{K\cap L}^{j} \in (\mathcal {N}^j  \cap \mathcal{F}_{j,o}(K \cup L)) } h_{F_{K\cap L}^{j}}^p \nonumber \\
& \qquad \qquad \qquad +\sum_{1 \leq j \leq d-1} (-1)^j \sum_{F_K^j \cap F_L^j \in (\mathcal {N}^j \setminus \mathcal{F}_{j,o}(K \cup L)) } h_{F_K^j \cap F_L^j}^p.
\end{align}
Let $F_K ^j \cap F_L^j \in \mathcal {N}^j \setminus \mathcal{F}_{j,o}(K \cap L)$. Hence $F_K ^j \cap F_L^j$ is a $(j-1)$-face of both $K$ and $L$ that contains the origin. Also $F_K ^j \cap F_L^j$ is not a $(j-1)$-face of $K \cup L$. Hence $F_K ^j \cap F_L^j \in \mathcal {D}^{j-1} \setminus \mathcal{F}_{j-1,o}(K \cup L)$. That means $\mathcal {N}^j \setminus \mathcal{F}_{j,o}(K \cap L) \subset \mathcal {D}^{j-1} \setminus \mathcal{F}_{j-1,o}(K \cup L)$.
On the other hand, $\mathcal {D}^{j-1} \setminus \mathcal{F}_{j-1,o}(K \cup L) \subset \mathcal {N}^j \setminus \mathcal{F}_{j,o}(K \cap L)$. Indeed, for $F \in \mathcal {D}^{j-1} \setminus \mathcal{F}_{j-1,o}(K \cup L)$, there exist an $i \geq j$ such that $F \subset \text{relint}\, F_{K \cup L}^i$. Then $i=j$ since otherwise $F$ will be contained in the relative interior of an $(i-1)$-face of $K$ which is a contradiction for the fact that $F$ is a $(j-1)$-face of $K$. Hence
\begin{align}\label{102}
\mathcal {D}^{j-1} \setminus \mathcal{F}_{j-1,o}(K \cup L) = \mathcal {N}^j \setminus \mathcal{F}_{j,o}(K \cap L).
\end{align}
Combining (\ref{104}) with (\ref{102}), (\ref{103}) holds true since
\begin{align*}
0=&\sum_{1 \leq j \leq d -1} (-1)^j \sum_{F \in \mathcal{D}^{j}} h_{F}^p
- \sum_{1 \leq j \leq d-2} (-1)^j \bigg(\sum_{F \in (\mathcal {D}^{j} \setminus \mathcal{F}_{j,o}(K \cup L))} h_{F}^p +  \sum_{F \in (\mathcal {D}^j \cap \mathcal{F}_{j,o}(K \cup L))} h_{F}^p \bigg) \nonumber\\
& \qquad \qquad -(-1)^{d-1} \sum_{F \in (\mathcal {D}^{d-1} \cap \mathcal{F}_{d-1,o}(K \cup L))} h_{F}^p
\end{align*}
(since $\mathcal {D}^{d-1} \cap \mathcal{F}_{d-1,o}(K \cup L) = \mathcal {D}^{d-1}$ or $\mathcal {D}^{d-1} \cap \mathcal{F}_{d-1,o}(K \cup L) = \emptyset$).


%
%

For case (\rmnum{4}), set $M = K \cup L$. There exists a hyperplane $H$ through the origin such that $K = M \cap H^+$, $L = M \cap H^-$ and $K \cap L = M \cap H$. Note that $\dim M =d$, $\dim (M \cap H) =d-1$ and $M \cap H$ is a $(d-1)$-face of $M \cap H^+$ and $M \cap H^-$, respectively. For $1 \leq j \leq d-1$, it is easy to see that
\begin{align*}
\mathcal{F}_{j,o}(M) &= \{F_{(M \cap H^+)}^j \in \mathcal{F}_{j,o}(M)\} \cup \{F_{(M \cap H^-)}^j \in \mathcal{F}_{j,o}(M)\} \cup \{F_{(M \cap H^+)}^j \cup F_{(M \cap H^-)}^j \in \mathcal{F}_{j,o}(M)\}
\end{align*}
and
\begin{align*}
\mathcal{F}_{j-1,o}(M \cap H) &= \{F_{(M \cap H^+)}^j \cap F_{(M \cap H^-)}^j:  F_{(M \cap H^+)}^j \cup F_{(M \cap H^-)}^j \in \mathcal{F}_{j,o}(M)\}.
\end{align*}
For $F_{(M \cap H^+)}^j \cup F_{(M \cap H^-)}^j \in \mathcal{F}_{j,o}(M)$, since
$$h_{F_{(M \cap H^+)}^j \cup F_{(M \cap H^-)}^j}^p + h_{F_{(M \cap H^+)}^j \cap F_{(M \cap H^-)}^j}^p = h_{F_{(M \cap H^+)}^j}^p + h_{F_{(M \cap H^-)}^j}^p,$$
we can check step by step that
\begin{align*}
&\sum_{1 \leq j \leq d-1} (-1)^j \sum_{F_{(M \cap H^+)}^{j} \in (\mathcal{F}_{j,o}(M \cap H^+) \setminus \{M \cap H \})} h_{F_{(M \cap H^+)}^{j}}^p \\
&\qquad \qquad + \sum_{1 \leq j \leq d -1} (-1)^j \sum_{F_{(M \cap H^-)}^{j} \in (\mathcal{F}_{j,o}(M \cap H^-)\setminus \{M \cap H \})} h_{F_{(M \cap H^-)}^{j}}^p \\
& =\sum_{1 \leq j \leq d-1} (-1)^j \sum_{F_{M}^{j} \in \mathcal{F}_{j,o}(M)} h_{F_{M}^{j}}^p
+ \sum_{1 \leq j \leq d -2} (-1)^j \sum_{F_{M \cap H}^j \in \mathcal{F}_{j,o}(M \cap H)} h_{F_{(M \cap H)}^j}^p.
\end{align*}
Now we only need to show that
\begin{align}\label{82}
\left(a_1 h_{M \cap H^+}^p + (a_2-a_1) h_{M \cap H}^p\right) + \left(a_1 h_{M \cap H^-}^p + (a_2-a_1) h_{M \cap H}^p\right) = a_1 h_M^p + (2a_2-a_1)h_{M \cap H}^p
\end{align}
if $d$ is odd, and
\begin{align}\label{83}
&\left((2a_2-a_1) h_{M \cap H^+}^p - (a_2-a_1) h_{M \cap H}^p\right) + \left((2a_2-a_1) h_{M \cap H^-}^p - (a_2-a_1) h_{M \cap H}^p\right) \nonumber \\
& \qquad \qquad = (2a_2-a_1) h_M^p + a_1 h_{M \cap H}^p
\end{align}
if $d$ is even. Indeed, (\ref{82}) and (\ref{83}) hold true since
$h_{M \cap H^+}^p + h_{M \cap H^-}^p = h_M^p + h_{M \cap H}^p.$
\end{proof}

For $a \in \mathbb{R}$, we write $a^p$ for $\text{sgn}(a) |a|^p$, where $\text{sgn}(a)= 1$ if $a \geq 0$, $\text{sgn}(a)= -1$ if $a < 0$.
\begin{pro}\label{pro5.1}
Let $0 \leq m \leq n$ and $v_0 \in \mathbb{R}^n$ be such that $o \in \text{\rm relint}~[v_0,e_1,\dots,e_m]$ and let $x = (x_1,\dots,x_d)^t \in \mathbb{R}^d$. Set $\alpha_1 = \max \{v_0 \cdot x,x_1,\dots,x_m \}$, $\alpha_2 = \min \{v_0 \cdot x,x_1,\dots,x_m \}$, $\beta_1 = \max \{ x_{m+1}, \dots, x_d\}$ and $\beta_2 = \min \{x_{m+1}, \dots, x_d \}$. Then
\begin{align}\label{90}
&\Phi_{p;a_1,a_2} ([v_0,e_1,\dots,e_d]) (x) \nonumber\\
&=a_2 \max \{\alpha_1^p, \beta_1 ^p \} +(a_2-a_1)(-1)^{m+1} \max \{\alpha_1^p ,\beta_2^p\} +(a_2-a_1)(-1)^m \alpha_1^p, \nonumber\\
&\Phi_{p;b_1,b_2} (-[v_0,e_1,\dots,e_d]) (x) \nonumber\\
&=b_2 \max \{-\alpha_2^p, -\beta_2 ^p \} +(b_2-b_1)(-1)^{m+1} \max \{-\alpha_2^p ,-\beta_1^p\} + (b_2-b_1)(-1)^m (-\alpha_2^p).
\end{align}
Especially, for $m=0$ and $v_0=o$,
\begin{align}\label{84}
\Phi_{p;a_1,a_2} (T^d) (x) &= a_2 \max \{\beta_1^p ,0\} - (a_2-a_1) \max \{\beta_2^p,0\}, \nonumber \\
\Phi_{p;b_1,b_2} (-T^d) (x) &= b_2 \max \{-\beta_2^p,0\}- (b_2-b_1) \max \{-\beta_1^p,0\}.
\end{align}
Moreover,
\begin{align}\label{85}
\Phi_{p;a_1,a_2} (T^d) (e_1) + \Phi_{p;b_1,b_2} (-T^d) (e_1) = a_2, \nonumber \\
\Phi_{p;a_1,a_2} (T^d) (-e_1) + \Phi_{p;b_1,b_2} (-T^d) (-e_1) = b_2
\end{align}
for $d \geq 2$, and
\begin{align}\label{86}
\Phi_{p;a_1,a_2} (T^1) (e_1) + \Phi_{p;b_1,b_2} (-T^1) (e_1) = a_1, \nonumber \\
\Phi_{p;a_1,a_2} (T^1) (-e_1) + \Phi_{p;b_1,b_2} (-T^1) (-e_1) = -b_1
\end{align}
for $d=1$.
\end{pro}
\begin{proof}
We will use the following basic equalities for binomial coefficients.
\begin{align}\label{bc1}
\sum_{m+1 \leq j \leq d-1} (-1)^j \left(\begin{array}{c} d-m-1 \\ j-m-1 \end{array} \right)  = (-1)^{d-1} ,
\end{align}
\begin{align}\label{bc2}
\sum_{m+1 \leq j \leq d-i+m+1} (-1)^j \left(\begin{array}{c} d-i \\ j-m-1 \end{array} \right) = 0, ~~m+2 \leq i \leq d-1.
\end{align}

Since $[v_0,e_1,\dots,e_d]$ is invariant under permutations of $\{e_{m+1},\dots,e_d \}$ and $\Phi_{p;a_1,a_2}$ is $\gln$ covariant, we can assume w.l.o.g. that $x_{m+1} \geq \dots \geq x_d$.
For $j < m$, $\mathcal{F}_{j,o}([v_0,e_1,\dots,e_d]) = \emptyset$. For $j =m$, $\mathcal{F}_{j,o}([v_0,e_1,\dots,e_d]) = \{ [v_0,e_1,\dots,e_m] \}$.
For $m+1 \leq j \leq d-1$,
\begin{align*}
\mathcal{F}_{j,o}([v_0,e_1,\dots,e_d]) = \big\{ [v_0,e_1,\dots,e_m,e_{\sigma_{m+1}},\dots,e_{\sigma_j}]: \{\sigma_{m+1},\dots,\sigma_j \} \subset \{m+1,\dots,d\} \big\},
\end{align*}
and
\begin{align*}
\sum_{F \in \mathcal{F}_{j,o}([v_0,e_1,\dots,e_d])} h_{F}^p (x) &= \left( \left(\begin{array}{c} d-m-1 \\ j-m-1 \end{array} \right) \max \{\alpha_1^p,x_{m+1}^p\} + \left(\begin{array}{c} d-m-2 \\ j-m-1 \end{array} \right) \max \{\alpha_1^p,x_{m+2}^p\} \right. \\
&\qquad \left.  \qquad + \dots + \left(\begin{array}{c} j-m-1 \\ j-m-1 \end{array} \right) \max \{\alpha_1^p,x_{d-j+m+1}^p\} \right).
\end{align*}
Hence, the definition of $\Phi_{p;a_1,a_2}$, (\ref{bc1}) and (\ref{bc2}) show that
\begin{align*}
\Phi_{p;a_1,a_2} ([v_0,e_1,\dots,e_d]) (x) &=a_2 \max \{\alpha_1^p,x_{m+1}^p\} +(a_2-a_1)(-1)^{m+1} \max \{\alpha_1^p,x_{d}^p\} \\
&\qquad \qquad +(a_2-a_1)(-1)^m \alpha_1^p.
\end{align*}
Then the second equation of (\ref{90}) follows from
$$\Phi_{p;b_1,b_2} (-[v_0,e_1,\dots,e_d]) (x) = \Phi_{p;b_1,b_2} ([v_0,e_1,\dots,e_d]) (-x).$$

For $m=0$ and $v_0=o$, we have $\alpha_1=\alpha_2=0$. Hence (\ref{84}) holds true.

(\ref{85}) and (\ref{86}) follow directly from (\ref{84}).
\end{proof}

Second, we give a lemma on lower dimensional polytopes.

\begin{lem}\label{lem5.8}
Let $n \geq 3$. If $\Phi : \mathcal{P}_o ^n \to C_p(\mathbb{R}^n)$ is an $\sln$ covariant valuation, then there exist constants $a_1,a_2,b_1,b_2 \in \mathbb{R}$ such that
\begin{align*}
\Phi P = \Phi_{p;a_1,a_2}P + \Phi_{p;b_1,b_2}(-P)
\end{align*}
for every $P \in \mathcal {P}_o ^n$ with $\dim P \leq n-1$.
\end{lem}
\begin{proof}
By the $\sln$ covariance of $\Phi$, Lemma \ref{lem4.1} and the inclusion-exclusion principle, we only need to show that
\begin{align}\label{87}
\Phi T^d (x) = \Phi_{p;a_1,a_2}T^d (x) + \Phi_{p;b_1,b_2}(-T^d) (x), ~~x \in \mathbb{R}^d
\end{align}
for $d \leq n-1$.

Set $a_d = \Phi(T^{d})(e_1)$ and $b_d = \Phi(T^{d})(-e_1)$ for $d \leq n-1$.

For $0 < \lambda <1$, define $H_\lambda$, $\phi _1,\phi _2$ as in Section \ref{s2}. For $d \leq n-1$, since $\Phi$ is a valuation, we get that
\begin{align*}
\Phi(T^{d}) + \Phi(T^d \cap H_\lambda) = \Phi(T^d \cap H_\lambda ^-) + \Phi(T^d \cap H_\lambda ^+).
\end{align*}
Also since $\Phi$ is $\sln$ covariant,
\begin{align}\label{55a}
\Phi(T^{d}) (x) + \Phi(\hat{T}^{d-1}) (\phi _1 ^t x)
= \Phi(T^d) (\phi _1 ^t x) + \Phi(T^d) (\phi _2 ^t x),
\end{align}
where $x=(x_1, \dots, x_n)^t$, $\phi_1 ^t x = (\lambda x_1 + (1-\lambda)x_2, x_2, x_3,\dots,x_{n-1}, \frac{1}{\lambda} x_n)^t$ and $\phi_2 ^t x = (x_1, \lambda x_1 + (1-\lambda)x_2,x_3,\dots,x_{n-1},\frac{1}{\lambda}x_n)^t$.

For $3 \leq d \leq n-1$, taking $x=e_d$ in (\ref{55a}), by Lemma \ref{lem4.1} and the $\sln$ covariance of $\Phi$, we obtain that $a_{d} = a_{d-1}$.
Thus, we have
\begin{align}\label{57a}
a_{n-1} = \dots = a_{2}.
\end{align}
Similarly, taking $x= -e_d$ in (\ref{55a}), we get
\begin{align}\label{64a}
b_{n-1} = \dots = b_{2}.
\end{align}

Now we will prove the desired result by induction on the dimension $d$. Proposition \ref{pro5.1} and the $p$-homogeneity of $\Phi T^d$, $\Phi_{p;a_1,a_2}T^d$ and $\Phi_{p;b_1,b_2}(-T^d)$ show that (\ref{87}) holds true for $d=1$. Assume that (\ref{87}) holds true for $d-1$. Then we will show that (\ref{87}) holds true for $d$. We will prove this by induction on the number $m$ of coordinates of $x$ not equal to zero. By the $\sln$ covariance of $\Phi$, we can assume w.l.o.g. that $x=x_1 e_1+\dots+x_m e_m$, $x_1,\dots,x_m \neq 0$.

Proposition \ref{pro5.1}, relations (\ref{57a}) and (\ref{64a}) show that (\ref{87}) holds true for $m=1$. Assume that (\ref{87}) holds true for $m-1$.

For $x_1 > x_2 > 0$ or $0 > x_2 >x_1$, taking $x=x_1 e_1 + x_3 e_3 + \dots +x_m e_m$, $\lambda = \frac{x_2}{x_1}$ in (\ref{55a}), we get
\begin{align}\label{38a}
&\Phi(T^d)(x_1e_1 + x_3 e_3 + \dots +x_m e_m) + \Phi(\hat{T}^{d-1}) (x_2 e_1 + x_3 e_3 + \dots +x_m e_m) \nonumber \\
&= \Phi(T^d) (x_2 e_1 + x_3 e_3 + \dots +x_m e_m) + \Phi(T^d) (x_1 e_1 + \dots +x_m e_m)\}.
\end{align}
For $x_2 > x_1 > 0$ or $0 > x_1 >x_2$,
taking $x=x_2 e_2 + x_3 e_3 + \dots +x_m e_m$, $1-\lambda = \frac{x_1}{x_2}$, in (\ref{55a}), we get
\begin{align}\label{40a}
&\Phi(T^d) (x_2e_2 + x_3 e_3 + \dots +x_m e_m) + \Phi(\hat{T}^{d-1}) (x_1 e_1 + \dots +x_m e_m)  \nonumber \\
&= \Phi(T^d) (x_1 e_1 + \dots +x_m e_m) + \Phi(T^d) (x_1 e_2 +x_3 e_3 + \dots +x_m e_m).
\end{align}
For $x_1 >0 > x_2$ or $x_2 > 0 > x_1$, taking $0 < \lambda = \frac{x_2}{x_2 - x_1} <1$ and $x= x_1 e_1 + \dots +x_m e_m$ in (\ref{55a}), we get
\begin{align}\label{23a}
&\Phi(T^d)(x_1 e_1 + \dots +x_m e_m) + \Phi(\hat{T}^{d-1}) (x_2 e_2 + x_3 e_3 + \dots +x_m e_m)  \nonumber \\
&= \Phi(T^d) (x_2 e_2 + x_3 e_3 + \dots +x_m e_m) + \Phi(T^d)(x_1 e_1 + x_3 e_3 + \dots +x_m e_m).
\end{align}
Combined with the $\sln$ covariance of $\Phi$, (\ref{38a}), (\ref{40a}) and (\ref{23a}) show that $\Phi(T^d) (x_1 e_1 + \dots +x_m e_m)$ is uniquely determined by $\Phi(T^d) (y_1 e_1 + \dots + y_{m-1} e_{m-1})$, $y_1, \dots,y_{m-1} \neq 0$, and $\Phi(T^{d-1})$. Since $\Phi_{p;a_1,a_2}(T^d) + \Phi_{p;b_1,b_2}(-T^d)$ also satisfies the equations (\ref{38a}), (\ref{40a}) and (\ref{23a}), we get that (\ref{87}) holds true for $m$. The proof is complete.
\end{proof}

Finally, let $\Phi' P = \Phi P - \Phi_{p;a_1,a_2}P - \Phi_{p;b_1,b_2}(-P)$, $P \in \mathcal {P}_o ^n$. Hence $\Phi'$ is a simple $\sln$ covariant valuation. Here simple means that the valuation vanishes on lower dimensional bodies. Combined with the following classification of simple valuations by Haberl \cite{Hab12b} and Parapatits \cite{Par14b}, we finish the proof of Lemma \ref{lem5.4}.

\begin{lem}[Haberl \cite{Hab12b} and Parapatits \cite{Par14b}]\label{lem5.7}
Let $n \geq 3$ and $\Phi : \mathcal{P}_o ^n \to C_p(\mathbb{R}^n)$ be a simple $\sln$ covariant valuation. Assume further that, for every $y \in \mathbb{R}^n$, the function $s \mapsto \Phi(sT^n)(y)$ is bounded from below on some non-empty open interval $I_y \subset (0,+\infty)$. Then there exist constants $c_1,c_2 \in \mathbb{R}$ such that
\begin{align*}
\Phi P = c_1 h_{M_p^+ P} ^p + c_2h_{M_p^- P}^p
\end{align*}
for every $P \in \mathcal {P}_o ^n$.
\end{lem}
\end{proof}

\begin{proof}[\textbf{Proof of Theorem \ref{thm5.1}.}]
For $a_1,b_1,c_1,c_2 \geq 0$, clearly $P \mapsto c_1 h_{M_p^+ P} ^p + c_2h_{M_p^- P}^p + a_1 h_{P}^p + b_1 h_{-P}^p$ is a valuation satisfying all conditions. Hence we only need to show the necessity.

Let $\Phi$ be a valuation satisfying all the conditions of Theorem \ref{thm5.1}. Since $\Phi$ also satisfies all the conditions of Lemma \ref{lem5.4}, there exist constants $a_1,a_2,b_1,b_2,c_1,c_2 \in \mathbb{R}$ such that
\begin{align}\label{105}
\Phi P = c_1 h_{M_p^+ P} ^p + c_2h_{M_p^- P}^p + \Phi_{p;a_1,a_2}P + \Phi_{p;b_1,b_2}(-P)
\end{align}
for every $P \in \mathcal {P}_o ^n$. The main aim is to show that $a_1=a_2$ and $b_1=b_2$.

\begin{lem}\label{lem5.1}
Let $\Phi$ satisfies (\ref{105}).
Assume that $\Phi P$ is non-negative and $(\Phi P)^{1/p}$ is a sublinear function for all $P \in \mathcal {P}_o ^n$. Then $a_1,a_2,b_1,b_2,c_1,c_2 \geq 0$.
Moreover, if $n \geq 3$, $p > 1$ or $n \geq 4$, $p=1$, then
$$a_1 = a_2, ~ b_1 = b_2;$$
if $p=1$ and $n=3$, then
$$a_1 \leq a_2, ~~b_1 \leq b_2, ~~a_2 - a_1 \leq b_2, ~~b_2 - b_1 \leq a_2.$$
\end{lem}
\begin{proof}
From the definitions,
\begin{align}\label{88}
h_{\alpha M_p^+ P} ^p = \alpha ^{n+p} h_{M_p^+ P} ^p,~~h_{\alpha M_p^- P} ^p = \alpha ^{n+p} h_{M_p^- P} ^p,
\end{align}
and
\begin{align}\label{89}
\Phi_{p;a_1,a_2} (\alpha P) = \alpha^p \Phi_{p;a_1,a_2} P,~~\Phi_{p;b_1,b_2}(-\alpha P) = \alpha^p \Phi_{p;b_1,b_2} (-P).
\end{align}
for $\alpha > 0$ and $P \in \mathcal {P}_o^n$.
Also since $h_{M_p^+ T^n}^p(e_1) >0$, $h_{M_p^- T^n}^p(e_1) = 0$, if $c_1 <0$, then $\alpha^{-p} \Phi(\alpha T^n) (e_1) \to -\infty$ when $\alpha \to \infty$. It is a contradiction since $\Phi(\alpha T^n) (e_1) \geq 0$ for any $\alpha >0$. Hence $c_1 \geq 0$. Similarly we get $c_2 \geq 0$.

Define $h(x) := \lim\limits_{\alpha \to 0^+} \alpha^{-1}(\Phi(\alpha T^3) (x))^{1/p}$, $x\in \mathbb{R}^3$.
By (\ref{88}) and (\ref{89}), we have
\begin{align}\label{h}
0 \leq h = (\Phi_{p;a_1,a_2} T^3 + \Phi_{p;b_1,b_2}(-T^3))^{1/p}.
\end{align}

Let $0 \leq \mu \leq \lambda \leq 1$. By Proposition \ref{pro5.1}, we get that
\begin{align*}
  h(e_1 + \lambda e_2 + \mu e_3) = (a_2 + \mu^p(a_1-a_2))^{1/p}.
\end{align*}
Especially,
\begin{align*}
  0 \leq h(e_1 + e_2 + e_3) = a_1^{1/p},
\end{align*}
and
\begin{align*}
  0 \leq h(e_1 + e_2) = h(e_1 + e_3) = a_2^{1/p}.
\end{align*}
On the other hand, $h$ is sublinear since $h$ is the limit of sublinear functions. Hence taking $\frac{1}{2} \leq \lambda \leq 1$, we get
\begin{align*}
  &(a_2 + (1-\lambda) ^p(a_1-a_2))^{1/p} = h(e_1 + \lambda e_2 + (1-\lambda) e_3) \\
  &\leq h(\lambda e_1 + \lambda e_2) + h((1-\lambda) e_1 + (1-\lambda) e_3) =  a_2^{1/p}.
\end{align*}
Then $a_1 \leq a_2$.

Next we will prove $a_2 \leq a_1$ for $n \geq 3$ and $p >1$.

Since $h$ is sublinear, it is also a support function of a convex body, denoted by $K \subset \mathbb{R}^3$. Let $x_1,x_2 \in \mathbb{R}$. By (\ref{h}) and Proposition \ref{pro5.1}, we get that
\begin{align*}
  h_{K|\mathbb{R}^2} (x_1 e_1 + x_2 e_2) &= h_{K} (x_1 e_1 + x_2 e_2) \\
  &= (\Phi_{p;a_1,a_2} T^3 (x_1 e_1 + x_2 e_2) + \Phi_{p;b_1,b_2}(-T^3) (x_1 e_1 + x_2 e_2))^{1/p} \\
  &= (a_2 \max \{x_1^p,x_2^p,0\} + b_2 \max \{-x_1^p,-x_2^p,0\})^{1/p} \\
  &= h_{a_2^{1/p}T^2 +_p (b_2)^{1/p} (-T^2) } (x_1 e_1 + x_2 e_2).
\end{align*}
Hence $K|\mathbb{R}^2 = a_2^{1/p}T^2 +_p a_2^{1/p} T^2$. If $a_2^{1/p} e_1 \notin K$, then $K$ must contain a point $a_2^{1/p} e_1 + \alpha e_3$, $\alpha \neq 0$. However, by similar arguments, the orthogonal projection of $K$ onto the linear space spanned by $\{e_1, e_3 \}$ is $a_2^{1/p} [o,e_1,e_3] +_p (b_2)^{1/p} (-[o,e_1,e_3])$. This is a contradiction since $a_2^{1/p} e_1 + \alpha e_3 \notin a_2^{1/p} [o,e_1,e_3] +_p (b_2)^{1/p} (-[o,e_1,e_3])$ when $p>1$. Hence $a_2^{1/p} e_1 \in K$.
Together with Proposition \ref{pro5.1}, we have
\begin{align*}
  a_2^{1/p} = a_2^{1/p} e_1 \cdot (e_1+e_2+e_3) \leq h_{K} (e_1 + e_2 +e_3) = a_1^{1/p}.
\end{align*}

For $n \geq 4$ and $p = 1$, we use $[-e_1,e_1,\dots,e_4]$ to show that $a_2 \leq a_1$.

Setting $d=4$, $m=1$, $v_0 = -e_1$ in (\ref{90}), we have
\begin{align*}
&\Phi_{1;a_1,a_2} ([-e_1,e_1,\dots,e_4]) \left(\begin{array}{c} 1 \\ 3 \\ 3 \\2 \end{array}\right)
+ \Phi_{1;b_1,b_2} (-[-e_1,e_1,\dots,e_4]) \left(\begin{array}{c} 1 \\ 3 \\3 \\ 2 \end{array}\right) \\
&=\Phi_{1;a_1,a_2} ([-e_1,e_1,\dots,e_4]) \left(\begin{array}{c} 1 \\ 3 \\ 2 \\3 \end{array}\right)
+ \Phi_{1;b_1,b_2} (-[-e_1,e_1,\dots,e_4]) \left(\begin{array}{c} 1 \\ 3 \\2 \\ 3 \end{array}\right) \\
&= 3 a_2 + 2 (a_2-a_1) - (a_2-a_1) + b_2,
\end{align*}
and
\begin{align*}
&\Phi_{1;a_1,a_2} ([-e_1,e_1,\dots,e_4]) \left(\begin{array}{c} 2 \\ 6 \\ 5 \\5 \end{array}\right)
+ \Phi_{1;b_1,b_2} (-[-e_1,e_1,\dots,e_4]) \left(\begin{array}{c} 2 \\ 6 \\5 \\ 5 \end{array}\right) \\
&= 6 a_2 + 5 (a_2-a_1) - 2(a_2-a_1) + 2b_2.
\end{align*}
Also since $\Phi_{1;a_1,a_2} ([-e_1,e_1,\dots,e_4])+ \Phi_{1;b_1,b_2} (-[-e_1,e_1,\dots,e_4])$ is sublinear, we have
$$5(a_2-a_1) \leq 4(a_2-a_1).$$
Hence $a_2 \leq a_1$.

The proof for the restrictions on $b_1,b_2$ is similar.

Finally, for $p=1$, $n=3$, since $h_{M_p^+ T^2} = h_{M_p^- T^2} =0$, $\Phi_{p;a_1,a_2} T^2 + \Phi_{p;b_1,b_2} (-T^2)$ is sublinear.
Also, for $i=1,2$, Proposition \ref{pro5.1} shows that
\begin{align*}
\Phi_{p;a_1,a_2} T^2 (e_i) &+ \Phi_{p;b_1,b_2} (-T^2)(e_i) = a_2, \\
\Phi_{p;a_1,a_2} T^2 (-e_i) &+ \Phi_{p;b_1,b_2} (-T^2)(-e_i) = b_2, \\
\Phi_{p;a_1,a_2} T^2 (e_1+e_2) &+ \Phi_{p;b_1,b_2} (-T^2)(e_1+e_2) = a_1, \\
\Phi_{p;a_1,a_2} T^2 (-e_1-e_2) &+ \Phi_{p;b_1,b_2} (-T^2)(-e_1-e_2) = b_1.
\end{align*}
Hence
\begin{align*}
a_2 &= \Phi_{p;a_1,a_2} T^2 (e_1) + \Phi_{p;b_1,b_2} (-T^2)(e_1) \\
 &\leq \Phi_{p;a_1,a_2} T^2 (e_1+e_2) + \Phi_{p;b_1,b_2} (-T^2)(e_1+e_2) + \Phi_{p;a_1,a_2} T^2 (-e_2) + \Phi_{p;b_1,b_2} (-T^2)(-e_2) \\
 &= a_1 + b_2,
\end{align*}
and
\begin{align*}
b_2 &= \Phi_{p;a_1,a_2} T^2 (-e_1) + \Phi_{p;b_1,b_2} (-T^2)(-e_1) \\
 &\leq \Phi_{p;a_1,a_2} T^2 (-e_1-e_2) + \Phi_{p;b_1,b_2} (-T^2)(-e_1-e_2)  + \Phi_{p;a_1,a_2} T^2 (e_2) + \Phi_{p;b_1,b_2} (-T^2)(e_2) \\
 &= b_1 + a_2.
\end{align*}
The proof is complete.
\end{proof}

Since $a_1= a_2$ and $b_1 = b_2$, we get
\begin{align*}
\Phi_{p;a_1,a_2} P = a_1 h_P^p, ~~ \Phi_{p;b_1,b_2} (-P) = b_1 h_{-P}^p
\end{align*}
for every $P \in \mathcal {P}_o ^n$.
Hence the proof is complete and the restrictions for $a_1,b_1,c_1,c_2$ are given by Lemma \ref{lem5.1}.
\end{proof}

\begin{proof}[\textbf{Proof of Theorem \ref{thm1.7}}]
First we show that $\Phi_{1;a_1,a_2}P + \Phi_{1;b_1,b_2}(-P)$ for $\dim P \leq 3$ is a \text{support} function (under the restrictions on $a_1,a_2,b_1,b_2$). We will use following two lemmas.

\begin{lem}\text{ \rm \cite[Lemma 3.2.9]{Sch14}} \label{sch1}
Let $K, L \in \mathcal {K}^n$. If $L|V$ is a summand of $K|V$, for all $2$-dimensional linear subspaces $V$ in some dense subset of $Gr(n, 2)$, then $L$ is a summand of $K$.
\end{lem}

\begin{lem}\text{ \rm \cite[Theorem 3.2.11]{Sch14}}\label{sch2}
Let $P, K \in \mathcal {K}^n$, where $P$ is a polytope. Then $P$ is a summand of $K$ if and only if $F(K, u)$ contains a translate of $F(P, u)$ whenever $F(P, u)$ is an edge of $P$ ($u \in S^{n-1}$).
\end{lem}

Now let $P \in \mathcal {P}_o^3$. If $o \in \text{relint}~{P}$, then there is nothing to prove. Assume $o \in \text{relbd}~P$.

First let $\dim P =3$.
Notice that
\begin{align*}
\Phi_{1;a_1,a_2}P + \Phi_{1;b_1,b_2}(-P) &= a_1 h_P + (a_2-a_1) \sum_{F \in \mathcal {F}_o (P)}h_{F} - (a_2-a_1) \sum_{E \in \mathcal {E}_o (P)}h_{E} \\
&\qquad + b_1 h_{-P} + (b_2-b_1) \sum_{F \in \mathcal {F}_o (P)}h_{-F} - (b_2-b_1) \sum_{E \in \mathcal {E}_o (P)}h_{-E}
\end{align*}
is a support function if and only if $(a_2-a_1) \sum_{E \in \mathcal {E}_o (P)} E + (b_2-b_1) \sum_{E \in \mathcal {E}_o (P)} (-E) =:P_1$ is a summand of $a_1 P + (a_2-a_1) \sum_{F \in \mathcal {F}_o (P)} F + b_1 (-P) + (b_2-b_1) \sum_{F \in \mathcal {F}_o (P)} (-F)=:P_2$.
According to Lemma \ref{sch1} and \ref{sch2}, it is sufficient to show that $F(P_2|V,u)$ contains a translate of $F(P_1|V,u)$ for all $V$ in a dense set of $Gr(n,2)$, whenever $F(P_1|V,u)$ is an edge of $P_1 |V$. Here and in the following $u \in S^{n-1} \cap V$. Also we can assume that for different edges $E_1,E_2 \in \mathcal {E}_o (P)$, $E_1|V$ and $E_2|V$ does not lie on the same line.

Let $m$ be the cardinality of the set $\mathcal {F}_o (P)$. Since the pointwise limit of a support function is a support function, it does not change the desired result. Thus we can assume that every face in $\mathcal {F}_o (P)$ has two edges containing the origin.
Also every edge in $\mathcal {E}_o (P)$ belongs to two faces in $\mathcal {F}_o (P)$. Hence $P$ also has $m$ edges through the origin.
Now we can write $\mathcal {F}_o (P)=\{F_i\}_{i=1}^m$ and $\mathcal {E}_o (P)=\{E_i\}_{i=1}^m$ such that $E_i \subset F_i \cap F_{i+1}$ for any $1 \leq i \leq n$. Here we set $F_{m+1} = F_1$.

Since
\begin{align*}
&P_1|V = (a_2-a_1) \sum_{i=1}^m E_i|V + (b_2-b_1) \sum_{i=1}^m (-E_i|V), \\
&P_2|V = a_1 P|V + (a_2-a_1) \sum_{i=1}^m F_i|V + b_1 (-P|V) + (b_2-b_1) \sum_{i=1}^m (-F_i|V),
\end{align*}
and $F(K+L,u)=F(K,u) + F(L,u)$ for $K,L \in \mathcal {K}^n$, we only need to show that if $F(E_i|V,u)$ is a non-degenerate interval (hence $F(E_i|V,u) = E_i|V$), then $F(P_2|V,u)$ contains a translate of $(a_2-a_1) E_i|V + (b_2-b_1) (-E_i|V)$. We need to deal with two cases: \\
(i) $E_i|V$ is contained in the boundary of $P|V$, \\
(ii) the relative interior of $E_i|V$ is contained in the relative interior of $P|V$.

In case (i), $u$ is an outer normal vector of $P|V$ or an inner normal vector of $P|V$. If $u$ is an outer normal vector of $P|V$, then $E_i|V$ is contained in $F(F_i|V,u)$, $F(F_{i+1}|V,u)$ and $F(P|V,u)$.
Hence $(a_2-a_1)F(F_i|V,u) + (a_2-a_1) F(F_{i+1}|V,u) + a_1 F(P|V,u)$ contains a translate of $(a_2-a_1) E_i|V + (b_2-b_1) (-E_i|V)$ since $b_2-b_1 \leq a_2$. Also since $F(P_2|V,u)$ contains a translate of $(a_2-a_1)F(F_i|V,u) + (a_2-a_1) F(F_{i+1}|V,u) + a_1 F(P|V,u)$, we have that $F(P_2|V,u)$ contains a translate of $(a_2-a_1) E_i|V + (b_2-b_1) (-E_i|V)$. \\
If $u$ is an inner normal vector of $P|V$, then $E_i|V$ is contained in
$- F(-F_i|V,u)$, $-F(-F_{i+1}|V,u)$ and $-F(-P|V,u)$.
Similarly $F(P_2|V,u)$ contains a translate of $(b_2-b_1)F(-F_i|V,u) + (b_2-b_1) F(-F_{i+1}|V,u) + b_1 F(-P|V,u)$ which contains a translate of $(a_2-a_1) E_i|V + (b_2-b_1) (-E_i|V)$ since $a_2-a_1 \leq b_2$.

In case (ii), $E_i|V$ is contained in $F(F_i|V,u) \cap F(F_{i+1}|V,-u)$ or $F(F_i|V,-u) \cap F(F_{i+1}|V,u)$. Hence $F(P_2|V,u)$ contains a translate of
$(a_2-a_1) E_i|V + (b_2-b_1) (-E_i|V)$.

The proof for $\dim P =2$ is similar (and easier). For $\dim P=1$, there is nothing to prove.

Now we turn to the necessity. Since $P \mapsto h_{ZP}$ satisfies the conditions of Lemma \ref{lem5.4}, there exist constants $a_1,a_2,b_1,b_2,c_1,c_2 \in \mathbb{R}$ such that
\begin{align*}
h_{ZP}= c_1 h_{M^+ P}  + c_2h_{M^- P} + \Phi_{1;a_1,a_2}P + \Phi_{1;b_1,b_2}(-P)
\end{align*}
for every $P \in \mathcal {P}_o ^n$. The restrictions for $a_1,a_2,b_1,b_2,c_1,c_2$ are given by Lemma \ref{lem5.1}.
\end{proof}

If we just consider valuations defined on $\mathcal {T}_o^n$, then $D_{a_1,a_2,b_1,b_2}$ is a valuation even for $n \geq 4$.

\begin{thm}\label{thm1.5}
Let $n \geq 3$. The map $Z : \mathcal{T}_o ^n \to \mathcal {K}_o ^n$ is an $\sln$ covariant Minkowski valuation if and only if
there exist constants $a_1,a_2,b_1,b_2,c_1,c_2 \geq 0$ satisfying $a_1 \leq a_2$, $b_1 \leq b_2$, $a_2-a_1 \leq b_2$ and $b_2-b_1 \leq a_2$ such that
$$ZT = c_1 M^+ T + c_2 M^- T + D_{a_1,a_2,b_1,b_2} T$$
for every $T \in \mathcal{T}_o ^n$,
where
$$D_{a_1,a_2,b_1,b_2} T = [a_2 v_i - b_2 v_j, a_2v_i -(a_2 -a_1)v_j, (b_2-b_1)v_i - b_2 v_j : 1 \leq i,j \leq d]$$
for $T = [o,v_1,\dots,v_d]$, $2 \leq d \leq n$, and
$$D_{a_1,a_2,b_1,b_2} T = [-b_1v_1,a_1v_1]$$
for $T = [o,v_1]$. Here $o,v_1,\dots,v_d \in \mathbb{R}^n$ are affinely independent.
\end{thm}
\begin{proof}
First, we show that the support function of $D_{a_1,a_2,b_1,b_2} T$ defined in this theorem is $\Phi_{1;a_1,a_2} T + \Phi_{1;b_1,b_2} (-T)$
if $a_1,a_2,b_1,b_2$ satisfy all the conditions. Since $D_{a_1,a_2,b_1,b_2}$ and $\Phi_{1;a_1,a_2}$ are both $\gln$ covariant,
we only need to show that
\begin{align*}
h_{D_{a_1,a_2,b_1,b_2} T^d} (y) = \Phi_{1;a_1,a_2} T^d (y) + \Phi_{1;b_1,b_2} (-T^d) (y)
\end{align*}
for $y \in \mathbb{R}^n$. But from the definition of $D_{a_1,a_2,b_1,b_2}$ and from Lemma \ref{lem4.1} and Lemma \ref{lem5.3}, we have
\begin{align*}
h_{D_{a_1,a_2,b_1,b_2} T^d} (y | \mathbb{R}^d) &= h_{D_{a_1,a_2,b_1,b_2} T^d} (y), \\
\Phi_{1;a_1,a_2} T^d(y) + \Phi_{1;b_1,b_2} (-T^d) (y) &= \Phi_{1;a_1,a_2} T^d(y | \mathbb{R}^d) + \Phi_{1;b_1,b_2} (-T^d) (y | \mathbb{R}^d).
\end{align*}
Combined with the $\gln$ covariance of $D_{a_1,a_2,b_1,b_2}, \Phi_{1;a_1,a_2}$ again, we only need to show that
\begin{align}\label{94}
h_{D_{a_1,a_2,b_1,b_2} T^d} (x) = \Phi_{1;a_1,a_2} T^d(x) + \Phi_{1;b_1,b_2} (-T^d) (x)
\end{align}
for $x = (x_1,\dots,x_d)^t \in \mathbb{R}^d$ with $x_1 \geq \dots \geq x_d$. A simple calculation shows that
\begin{align}\label{92}
h_{D_{a_1,a_2,b_1,b_2} T^d} (x) &= \max_{1\leq i,j \leq d} \{a_2 x_i - b_2 x_j, a_2x_i -(a_2 -a_1)x_j, (b_2-b_1)x_i - b_2 x_j \} \nonumber \\
& = \max \{a_2 x_1 - b_2 x_d, a_2x_1 -(a_2 -a_1)x_d, (b_2-b_1)x_1 - b_2 x_d \}.
\end{align}
Also Proposition \ref{pro5.1} shows that
\begin{align}\label{93}
&\Phi_{1;a_1,a_2} (T^d) (x) + \Phi_{1;b_1,b_2} (-T^d)(x) \nonumber \\
=&a_2 \max \{x_1,0\} - (a_2-a_1) \max \{x_d,0\} + b_2 \max \{-x_d,0\}- (b_2-b_1) \max \{-x_1,0\}.
\end{align}
For all the three cases $0 \geq x_1 \geq x_d$, $x_1 \geq 0 \geq x_d$ and $x_1 \geq x_d \geq 0$, the right side of (\ref{92}) and (\ref{93}) is equal. Hence, (\ref{94}) holds true.

Since $T \mapsto c_1 h_{M^+ T} + c_2h_{M^- T} + \Phi_{1;a_1,a_2} T  + \Phi_{1;b_1,b_2} (-T)$ is a valuation so is $T \mapsto c_1 M^+ T + c_2 M^- T + D_{a_1,a_2,b_1,b_2} T$. The proof of the sufficient part is complete.

Next we turn to the necessity. Since $T \mapsto h_{ZT}$ satisfies the conditions of Lemma \ref{lem5.4}, there exist constants $a_1,a_2,b_1,b_2,c_1,c_2 \in \mathbb{R}$ such that
\begin{align*}
h_{ZT} = c_1 h_{M^+ T}  + c_2h_{M^- T} + \Phi_{1;a_1,a_2}T + \Phi_{1;b_1,b_2}(-T)
\end{align*}
for every $T \in \mathcal {T}_o ^n$ (Although the domain of the valuation is just $\mathcal{T}_o ^n$ not $\mathcal{P}_o ^n$, we still can get this result from the proof of Lemma \ref{lem5.4}). The restrictions on $a_1,a_2,b_1,b_2,c_1,c_2$ are given by Lemma \ref{lem5.1}.
\end{proof}

\section*{Acknowledgement}
The work of the authors was supported, in part, by the \text{National} \text{Natural} Science \text{Foundation} of China (11271244) and Shanghai Leading Academic Discipline Project (S30104). The first author was also supported by China Scholarship Council (CSC 201406890044).


\end{document}